\documentclass[a4paper,french,twoside,draft,english,11pt,centertags]{smfartVSdml}
\usepackage{smfenum,smfthm,amssymb,amsmath,amsthm,amscd,mathrsfs,euscript} 
\usepackage{stmaryrd,mathabx,upgreek,color}
\usepackage[latin1]{inputenc}
\input{xypic}

\numberwithin{equation}{section} 
\theoremstyle{plain}

\def\ZZ{\mathbf{Z}} 

\def\A{{\rm A}}
\def\B{{\rm B}}
\def\C{{\rm C}}
\def\D{{\rm D}}

\def\F{{\rm F}}
\def\G{{\rm G}}
\def\H{{\rm H}}

\def\K{{\rm K}}
\def\L{{\rm L}}
\def\M{{\rm M}}

\def\O{{\rm O}}
\def\P{{\rm P}}

\def\R{{\rm R}}
\def\SS{{\rm S}}
\def\T{{\rm T}}
\def\U{{\rm U}}
\def\V{{\rm V}}
\def\W{{\rm W}}
\def\X{{\rm X}}
\def\Y{{\rm Y}}
\def\Z{{\rm Z}}
\def\Ker{{\rm Ker}}

\def\Cc{\mathscr{C}}

\def\Oo{\mathscr{O}}

\def\a{\alpha} 
\def\b{\beta}

\def\l{\lambda}
\def\m{\mathfrak{m}}
\def\n{\mathfrak{n}}
\def\o{\EuScript{O}}
\def\p{\mathfrak{p}}
\def\s{\sigma}

\def\w{\varpi}

\def\Om{\Omega}

\def\LL{\Lambda}
\def\MM{\Pi}
\def\VV{\V}

\def\rp{\rangle}
\def\>{\geqslant}
\def\<{\leqslant}

\def\tdt{\times\dots\times}

\def\Hom{{\rm Hom}}

\def\GL{{\rm GL}}

\def\Ker{{\rm Ker}}

\def\Ind{{\rm Ind}}
\def\ind{{\rm ind}}

\def\mult#1{{#1}^{\times}}

\def\Cc{\EuScript{C}}

\def\Oo{\EuScript{O}}

\def\ip{\boldsymbol{i}}
\def\rp{\boldsymbol{r}}

\def\sy#1{\boldsymbol{[}#1\boldsymbol{]}}
\def\({\left(}
\def\){\right)}

\def\ffr#1{\smash{\mathop{\ \longrightarrow\ }\limits^{#1}}}

\def\St{{\rm St}}
\def\GH{\widehat{\G}}
\def\UIQ{\textbf{\textsf{Q}}}
\def\UIS{\textbf{\textsf{S}}}
\def\1{1}

\def\vee{*}
\def\ii{\infty}
\def\qc{f}
\def\ee{e}

\author{V.~Sécherre}
\address{Université de Versailles Saint-Quentin-en-Yvelines\\
Laboratoire de Mathémati\-ques de Versailles\\
45 avenue des Etats-Unis\\
78035 Versailles cedex, France}
\email{vincent.secherre@math.uvsq.fr}

\author{C.~G.~Venketasubramanian}
\address{Department of Mathematics\\
Ben-Gurion University of the Negev\\
P.O.B 653, 
Beer Sheva 84105, Israel.} 
\email{coolimut@math.bgu.ac.il}

\title[Modular representations of $\GL{(n)}$ distinguished by $\GL{(n-1)}$]
{Modular representations of $\GL{(n)}$ distinguished by $\GL{(n-1)}$ over a
  $p$-adic field}

\begin{abstract}
Let $\F$ be a non-Archimedean locally compact field, $q$ be the 
cardinality of its residue field, and $\R$ be an algebraically closed 
field of characteristic $\ell$ not dividing $q$.
We classify all irredu\-cible smooth $\R$-representations of $\GL_n(\F)$ 
having a nonzero $\GL_{n-1}(\F)$-inva\-riant linear 
form, when $q$ is not congruent to $1$ mod $\ell$.
Partial results in the case when $q$ is $1$ mod $\ell$ show that, 
unlike the complex case, 
the space of $\GL_{n-1}(\F)$-invariant linear forms has dimension $2$ for certain 
irreducible representations.
\end{abstract}

\thanks{The authors have been partially supported by the 
Agence Natio\-nale de la Recherche grant ANR-08-BLAN-0259-01. 
The second named author was also partially supported by a 
PBC post-doctoral fellowship, 
Israel, a post-doctoral fellowship funded by the Skirball Foundation via the 
Center for Advanced Studies in Mathematics at Ben-Gurion University of the 
Negev and also by ISF grant 1138/10.}

\long\def\MSC#1\EndMSC{\def\arg{#1}\ifx\arg\empty\relax\else
     {\par\narrower\noindent%
     2010 Mathematics Subject Classification: #1\par}\fi}

\long\def\KEY#1\EndKEY{\def\arg{#1}\ifx\arg\empty\relax\else
	{\par\narrower\noindent Keywords and Phrases: #1\par}\fi}

\begin{document}

\maketitle

\MSC 
22E50 
\EndMSC

\KEY 
Modular representations of $p$-adic reductive groups, 
Distinguished representations, Gelfand pairs
\EndKEY

\section{Introduction}

\subsection{}

Let $\F$ be a non-Archimedean locally compact field of residual 
char\-ac\-teristic $p$,
let $\G$ be the 
$\F$-points of a connected reductive group 
over $\F$ together with a closed subgroup $\H$ of $\G$, 
and let $\R$ be an algebraically closed field of characteristic different from $p$. 
Given irreducible smooth rep\-resentations $\pi$ of $\G$ and $\s$ of $\H$
with coefficients in $\R$, it is a question of
ge\-ne\-ral interest in representation theory, known as the branching problem, to
understand whe\-ther $\pi$ restricted to $\H$ has $\s$ as a quotient.
If $\R$ is the field of complex numbers, this question is classical and well
understood is many situations (see for instance \cite{GGP1,GGP2}).
A case of particular interest is when $\s$ is the trivial representation. 
In this situation, the representation $\pi$ is said to be 
$\H$-distinguished if its restriction to $\H$ has the trivial representation 
as a quotient, that is, if $\pi$ carries a nonzero $\H$-invariant linear 
form.

\subsection{}

In this article, we are interested in the case where $\G$ is the general 
linear group $\GL_n(\F)$, with $n\>2$, and $\H$ is the group 
$\GL_{n-1}(\F)$ embedded in $\G$ via:
\begin{equation*}
\label{EMBED}
x\mapsto
\begin{pmatrix}
x&0\\0&1
\end{pmatrix}.
\end{equation*}
When $\R$ is the field of complex numbers, it is a consequence of a result of 
Waldspurger~\cite{W} that, for $n=2$,
any infinite dimensional irreducible representation of 
$\G$ is $\H$-distinguished. 
The clas\-si\-fication of all $\H$-distinguished irreducible 
representations of $\G$ for $n=3$
has been done by D. Prasad~\cite{P}. 
For any $n\>2$, Prasad~\cite{P} has also proved that any generic representation 
of $\G$ has every generic representation of $\H$ as a quotient, 
and Flicker~\cite{Flicker} classified all $\H$-distinguished irreducible 
unitary representations of $\G$.
The classification of all $\H$-distinguished irreducible rep\-re\-sentations of $\G$ 
for any integer $n\>3$ has been obtained by Venket\-asubra\-ma\-nian~\cite{V}, 
in terms of Langlands parameters.
Thus, when $\R$ is the field of complex numbers, the question is well under\-stood. In this paper, we investigate the case where the field $\R$ 
has positive characteristic $\ell$ different from $p$. 

\subsection{}

The representation theory of smooth representations of $\GL_n(\F)$
with coefficients in any alge\-braically closed field $\R$ of characteristic 
$\ell\neq 0,p$ has been initiated by Vignéras~\cite{Vigb,Vigs}
in view to extend the local Langlands program to representations with 
coefficients in a field (or ring) as ge\-ne\-ral as possible 
(see for instance \cite{Vigl}).
It has then been pursued by 
Dat, M\'\i nguez, Stevens and the first author \cite{Datf,MSb,MSu,MSc,MSt,SSb}.
In many aspects, it is similar to the theory of com\-plex representations of 
this group:
the fact that $\ell$ is different from $p$ ensures that there is an 
$\R$-valued Haar measure on $\G$, the functors of parabolic induction and 
restriction are exact and preserve finite length, 
there is a theory of derivatives, 
there is a notion of cuspidal support for irre\-du\-cible represen\-ta\-tions 
and a classification of these representations by mutisegments. 
How\-ever, there are also im\-portant differences:
the measure of a compact open subgroup may be $0$,
and the notions of cuspidal and supercuspidal representations do not 
coinci\-de, that is, a representation whose all proper 
Jacquet modules are zero 
may occur as a subquo\-tient of a proper parabolically induced representation. 
The combinatorics of multisegments is also much more involved, 
since the cardi\-nality $q$ of the residue field of $\F$ 
has finite order in $\mult\R$.

\subsection{}

We now come to the main theorem of this article. 
Let $\R$ denote an algebraically closed field of char\-ac\-teristic different from 
$p$ (possibly $0$) and write $e$ for the order (possibly infinite) of $q$ in 
$\mult\R$. 
Write $\nu$ for the normalized absolute value of $\F$ giving value 
$q^{-1}$ to any uniformizer. 
Let us fix a square root of $q$ in $\R$, denoted $\sqrt{q}$.
Given integers $k\in\ZZ$ and $n\>1$, we write:
\begin{equation*}
\nu_{n}^{k/2}:g\mapsto(\sqrt{q})^{-k\cdot{\rm val}(\det(g))}
\end{equation*}
where ${\rm val}$ is the normalized valuation on $\F$ and 
$\det$ is the determinant from $\G$ to $\mult\F$. 
If $\pi,\s$ are smooth representations of $\GL_u(\F)$, $\GL_v(\F)$
respectively, with $u+v=n$,
we denote by $\pi\times\s$ the normalized parabolic 
induction of $\pi\otimes\s$ to $\G$ along the standard (upper triangular)
parabolic subgroup. 
When $e>1$, the induced representation:
\begin{equation}
\label{INDUITEFONDAMENTALE}
\V_n^{} = \nu_{n-1}^{1/2}\times\nu^{(n+1)/2}
\end{equation} 
has a unique irreducible quotient, denoted $\LL_{n}$
(see Example \ref{poutargue}).
Note that, when $e$ divides $n$, this representation is the trivial character. 
Let us write $\1_n$ for the trivial representation of $\G$.

\begin{theo}
\label{MAINTHEOREM}
Suppose that $n\>2$ and $e>1.$
An irreducible representation of $\GL_n(\F)$ is $\GL_{n-1}(\F)$-distingui\-shed 
if and only if it belongs to the following list: 
\begin{enumerate}
\item
the trivial representation $\1_n$;
\item
an irreducible representation of the form $\nu_{n-1}^{\pm1/2}\times\chi$ 
with $\chi$ a character of $\GL_1(\F)$;
\item
an irreducible representation of the form $\1_{n-2}\times\tau$ 
with $\tau$ an infinite dimensional irreducible representation of $\GL_2(\F)$;
\item
the representation $\LL_n$ and its contragredient.
\end{enumerate}
\end{theo}

As in the complex case,
the proof of Theorem \ref{MAINTHEOREM} is by induction on $n$. 
There are two parts to the proof of Theorem \ref{MAINTHEOREM}: 
proving that the representations in the list offered by the theorem 
are $\H$-distinguished is the easier part.
The more difficult part is to show the converse, 
namely that all irreducible representations which are
$\H$-distinguished are in the list. 

\subsection{}

Since our proof is by induction, we first treat the case when $n=2$ and obtain
a classification (see Theorem \ref{dimfor2}) of all the 
$\GL_1(\F)$-distinguished irredu\-ci\-ble representations of $\GL_2(\F)$. 
When $e$ is not $1$, the result turns out to be the same as in the complex 
case: all infinite-dimensional irreducible representation of $\GL_2(\F)$ are 
distinguished and their space of $\GL_1(\F)$-invariant linear forms has 
dimension $1$. 
When the characteristic of $\R$ divides $q-1$ however, 
this dimension is $2$ for certain representations. 

\subsection{}

Assume now that $n\>3$.
As in the complex case, one can show by restricting to the mirabolic subgroup 
that none of the cuspidal representations of $\G$ are distinguished
(Theorem \ref{THEOCUSP}). 
Since any non-cuspidal irreducible
smooth representation of $\G$ is a quotient of a
parabolically induced representation of the form $\sigma\times\tau$ with
$\sigma,\tau$ smooth irreducible representations of $\GL_u(\F), \GL_{v}(\F)$
for some integers $u,v\>1$ such that $u+v=n$, 
it is natural to study the distinction of $\sigma\times\tau$. 
This was carried out in \cite{V} in the complex case.
In the modular case, it works as in the complex case:
one gets a set of three necessary
conditions for this induced representation to be $\H$-distinguished, 
of which two are sufficient (see Lemma \ref{TOL}). 
This is attributed to the exis\-ten\-ce of three orbits for the action of $\H$ on
the homogeneous space made of all subspaces of dimension $u$ in $\F^n$,
out of which two are closed.
The induced representations in (2) and (3) of Theorem 
\ref{MAINTHEOREM} are shown to satisfy one of the sufficiency conditions 
coming from Lemma \ref{TOL} (see Corollary \ref{Converse}). 
The contragredient of $\LL_n$, when non-trivial, is 
realized as a subrepresentation of a distinguished principal series of length 
2, the quotient of which is a nontrivial character and is non-distinguished 
(see proof of Lemma \ref{Ldistchar}).

\subsection{}

To prove the converse of Theorem \ref{MAINTHEOREM}, we first prove that any
$\H$-distinguished representation of $\G$ is a quotient of a representation of
the form $\rho\times \chi$ where $\rho$ is an irreducible representation of
$\GL_{n-1}(\F)$ and $\chi$ a character of $\mult\F$. 
\textit{In particular, when $e>1$, such a quotient is unique.}
Using the conditions of 
Lemma \ref{TOL} mentioned above and the induction hypothesis, we can specify 
$\rho$ and $\chi$ to be in a list (see Proposition \ref{reduction1}). 
Then, when $e>1$, we analyze the unique irreducible quotient of all these 
$\rho\times\chi$. 
We show that if the quotient is distinguished, then it must be in our list. 
The case when $e=1$
presents additional difficulties which we shall touch open below. 

\subsection{}

We now describe the contents of the article. 
In Section \ref{notation} we set some basic notation, 
and deal with the case $n=2$ in Section \ref{GL2}. 
The complete classification for $n=2$ is obtained in Theorem \ref{dimfor2}. 
We begin Section \ref{modlresults} by recalling 
some general results on $\ell$-modular representations of $\GL_n(\F)$
from \cite{Vigb,MSc}. 
We get a complete description of the subquotients of 
representations of the form 
$\Z(\Delta)\times \Z(\Delta')$ where $\Delta,\Delta'$ are segments and 
$\Delta'$ is of length at most $2$ 
(see Propositions \ref{P61} and \ref{length2seg}). 
In particular, Proposition \ref{P61} may be deemed to be a generalization of 
\cite[Théorème~3]{VigComp}. 
Moreover, comparing with \cite[Proposition 4.6]{Z}, 
Propositions \ref{P61} and \ref{length2seg} highlight one of the essential 
differences between principal series representations in the complex 
and modular cases: 
a product of two characters has length at most 2 in the complex case, 
a fact which does not hold as such in the modular case. 
The representation $\Lambda_n$, 
which plays a essential role in the article, 
is defined in Example \ref{poutargue} for $e>1$, and in Definition
\ref{poutargue4} in general. 
More generally, 
Section \ref{Castor} is devoted to the case where $e=1$.
The avatar $\Pi_n$ of $\LL_n$ is defined in Example \ref{poutargue2}. 
In Section \ref{CompDer}, we compute the derivatives of $\Lambda_n$ and $\Pi_n$. 

In Section \ref{blmanalog}, we prove a criterion for 
irreducibility of a  product of the form  $\Z(\Delta)\times \L(\Delta')$ where
$\Delta'$ has length 2. 
This is a modular version of a result known in the
complex case (Theorem 3.1 in \cite{BLM}). 
We begin Section \ref{distprelims} with 
some basic results on $\H$-distinguished representations of $\G$. 
The first tool is to use Lemma \ref{TOL}, the conditions that we get from the three
orbits that we mentioned above. 
This, along with some of its consequences, 
yields us Proposition \ref{reduction1} and we get a list of representations of
the form $\rho\times \chi$ (see the list following Proposition 
\ref{reduction1}): understanding the distinction of the quotients of 
representations in this list proves the difficult part of Theorem 
\ref{MAINTHEOREM}. 
The second tool in our proof is Proposition \ref{BZapp} 
using the Bernstein-Zelevinski filtration, which was available for the complex
case \cite{Flicker,P} and holds for $\R$. 
The computation of the quotients of $\rho\times \chi$ in the list obtained 
from Proposition \ref{reduction1} is the content of Sections 9-12. 

\subsection{}

We now explain some  of the subtler ideas behind our  proof which is different
from the one in \cite{V} proved  for complex representations. In \cite{V}, the
main tool in  analyzing the existence of a unique  irreducible quotient is the
Langlands Quotient Theorem and certain results of Zelevinski \cite{Z}. 
When these theorems fail to apply, \cite{V} uses Theorem 7.1 of \cite{Z}. 
In fact, we use Pro\-po\-si\-tion \ref{UIQ} which is sufficient for us to analyze the 
representations coming from the Lemma \ref{TOL} when $e>1$. 
Indeed, if one
were to use Proposition \ref{UIQ} in the complex case, then the 
proof of Theorem \ref{MAINTHEOREM} in \cite{V} simplifies to some extent 
without having to resort to Theorem 7.1 of \cite{Z}, because there we have to 
analyze all subquotients of a certain induced representation. 

\subsection{}

However, in the modular case, 
even if Proposition \ref{UIQ} 
guarantees the existence of a unique ir\-re\-ducible quotient for 
the representations $\rho\times\chi$ arising from Proposition \ref{reduction1}, 
to explicitly find this quotient is more difficult. 
This is due to the fact that, 
in order to determine whether the unique irreducible quotient of $\rho\times \chi$ is in 
the list offered by Theorem \ref{MAINTHEOREM}, we have to realize it as a 
quotient of a larger principal series and this larger principal series may not 
have a unique irreducible quotient. 
In such a situation, we had but no choice 
to use the analogue of Theo\-rem 7.1 of \cite{Z} for the larger principal 
series in hand. 
For our purposes, we reduce it to understand the subquotients 
of representations of the form $\Z(\Delta)\times\Z(\Delta')$ where the segment 
$\Delta'$ has length at most $2$. 
These subquotients have certain natural properties (see Section 
\ref{modlresults}, $\bf{P1}$ to $\bf{P6}$) proved in \cite{MSc} which enables us 
to describe the subquotients. 
This result is obtained in Propositions \ref{P61} and \ref{length2seg}. 
We then use Proposition \ref{BZapp} to rule out the subquotients in the larger
principal series which are not in the list of Theorem \ref{MAINTHEOREM}. 

\subsection{}

Let us mention that, when $e=1$, the list in Theorem \ref{MAINTHEOREM} is
not exhaustive. The first problem is that the representation 
\eqref{INDUITEFONDAMENTALE} need not have a unique irreducible 
quotient. 
In particular, all its irreducible subquotients are
$\H$-distinguished (see Lemma \ref{Pi_ndist}), which is different behavior when
we compare with the case when $e>1$. This forces us to consider more
representations in the list offered by Proposition \ref{reduction1} and the
tools that we use does not seem to be sufficient to understand the distinction
of the quotients. 

\subsection{}

We now give a few remarks. 
First, the theory of $p$-modular representations of $p$-adic reductive groups 
is very different from the $\ell$-modular theory. 
This is why we have chosen to focus on the case where $\R$ has characteristic 
different from $p$.

In the complex case, the pair $(\G,\H)$ is known to be a 
Gelfand pair (see \cite{AGS}).
This is no longer true in the modular case: when $e$ is equal to $1$, some 
$\H$-distinguished have a $2$-dimensional space of $\H$-invariant linear 
forms
(see Theorem \ref{gl2mult} and Remark \ref{REMARQUEFINALE}).

When com\-paring the results in \cite{V} with M\'\i nguez \cite{M}, 
the clas\-si\-fication of all $\H$-distinguished irreducible complex 
rep\-re\-sentations of $\G$ turns out to be easily expressed in terms 
of the local theta correspondence from $\GL_2(\F)$ to $\GL_n(\F)$. 
It would be interesting to investigate this in the modular case, by developing 
an $\ell$-modular theta correspondence (see \cite{Ml}). 

\subsection*{Acknowledgements}

This work was conceived when the second named author was a Post Doctoral 
Fellow (CNRS) at Laboratoire de Mathematiques de Versailles, France, during
October 2011-September 2012. 
Some parts of this work were done while the second named author was a Visiting 
Fellow at TIFR, Mumbai during October-November 2012 and he wishes to thank 
Prof. Dipendra Prasad for the in\-vitation. 
He wishes to thank the above
organizations as well as Department of Mathematics at Ben-Gurion University of 
the Negev for extending financial support and excellent fa\-ci\-lities.

\section{Notation and preliminaries}\label{notation}

In all this article, we fix a locally compact non-Archimedean field $\F$~; 
we write $\o$ for its ring of integers, $\p$ for the maximal ideal of $\o$ 
and $q$ for the cardinality of its residue field. 
We also fix an algebraically  closed field $\R$ of characteristic not dividing
$q$.

We write $e$ for the order (possibly infinite) of the image of $q$ in 
$\mult\R$ and define:
\begin{equation*}
\qc=
\left\{
\begin{array}{l}
0 \text{ if $\R$ has characteristic $0$},\\
\text{the smallest positive integer $k\>2$ such that 
$1+q+\dots+q^{k-1}=0$ in $\R$ otherwise}. 
\end{array}
\right.
\end{equation*}
When $\R$ has characteristic $\ell>0$, 
we have $\qc=e$ if $e>1$ and $\qc=\ell$ if $e=1$. 

Given a topological group $\G$, a smooth $\R$-repre\-sen\-tation 
(or representation for short)
of $\G$ is a pair 
$(\pi,\V)$ made of an $\R$-vec\-tor spa\-ce $\V$ together with a group 
homo\-morphism 
$\pi:\G\to\GL(\V)$ such that, for all $v\in\V$, there is an open subgroup 
of $\G$ fixing $v$.
\textit{In this article, all representations will be supposed to be smooth
  $\R$-representations.} 

A smooth $\R$-character (or character for short) of $\G$ is 
a group homomorphism from $\G$ to $\mult\R$ with open kernel. 

Given a representation $\pi$ and a character $\chi$ of $\G$,
we write $\pi\chi$ for the twisted representation $g\mapsto\pi(g)\chi(g)$.

For $n\>1$, we write $\G_n=\GL_n(\F)$, and $\GH_n$ for the set of isomorphism 
classes of its irreducible repre\-sentations.
In particular, $\GH_1$ will be identified with the group of characters of 
$\G_1$. 

Given a representation $\pi$ of $\G_{n}$, $n\>1$ and $\mu\in\GH_1$, 
we write $\pi\cdot\mu=\pi(\mu\circ\det)$.
If $\pi$ has finite length, we write $[\pi]$ for its semi-simplification. 

\section{The pair $(\GL_2(\F),\GL_1(\F))$}
\label{GL2}

Write $\G=\GL_2(\F)$ and let:
\begin{equation*}
\H=\left\{
\begin{pmatrix}
x&0\\0&1
\end{pmatrix}\ ;\ x\in\mult\F \right\}\subseteq\G.
\end{equation*}

Let $\B$ denote the Borel subgroup of $\G$ made of upper triangular matrices, 
and write:
\begin{equation*}
s=\begin{pmatrix}0&1\\1&0\end{pmatrix}\in\G.
\end{equation*}

If $\X$ is a locally compact topological space and $\A$ is a commutative ring, 
let $\Cc^{\ii}_{c}(\X,\A)$ denote the space of all locally constant 
and compactly supported functions from $\X$ to $\A$. 

We write $dx$ for the $\R$-valued Haar measure on $\mult\F$ giving measure 
$1$ to the subgroup $1+\p$ of principal units (see \cite[I.2]{Vigb}).

\subsection{The principal series}
\label{PS}

Let $\a_1,\a_2$ be two smooth $\R$-characters of $\mult\F$.
Let:
\begin{equation*}
\V=\V(\a_1,\a_2)
\end{equation*}
denote the (non-normalized) parabolic $\R$-induction 
$\Ind_\B^\G(\a_1\otimes\a_2)$, that is the space of all locally constant 
$\R$-valued functions $f$ on $\G$ such that $f(mng)=\a_1(m_1)\a_2(m_2)f(g)$ 
for all:
\begin{equation*}
m=
\begin{pmatrix}
m_1&0\\
0&m_2
\end{pmatrix}
\in\G,
\quad
n\in 
\begin{pmatrix}
1&\F\\
0&1
\end{pmatrix}
\subseteq\G,
\quad
g\in\G,
\end{equation*}
which is made into a smooth $\R$-representation of $\G$ by making $\G$ act by 
right translations. 
Write $\W$ for the sub\-space of $\V$ made of all functions vanishing 
at $1$ and $s$.
The map:
\begin{equation*}
\W\to\Cc^\infty_c(\mult\F,\R)
\end{equation*}
which associates to $f\in\W$ the function:
\begin{equation*}
\phi:x\mapsto f\left(s 
\begin{pmatrix}1&x\\0&1\end{pmatrix}
\right)
\end{equation*}
is an isomorphism of $\R$-vector spaces, and becomes an isomorphism 
of representations of $\H$ if the right hand side is endowed with the action 
defined by:
\begin{equation*}
a\cdot\phi:x\mapsto\a_2(a)\phi(xa^{-1}),
\quad
x,a\in\mult\F. 
\end{equation*}
Up to a nonzero scalar, there is on $\W$ a unique nonzero $\H$-invariant 
linear form, given by: 
\begin{equation*}
\mu:f\mapsto\int\limits_{\mult\F}
f\left(s\begin{pmatrix}1&x\\0&1\end{pmatrix}\right)
\a_2(x)^{-1}\ dx.
\end{equation*}

Let $\a$ denote the character of $\B$ extending $\a_1\otimes\a_2$.  
Fix an integer $i\>1$ such that $\a_1,\a_2$ are trivial on $1+\p^i$, 
and let $\K_i$ be the subgroup of $\GL_2(\o)$ made of matrices 
congruent to the identity mod $\p^i$. 
We define two functions $f_0$ and $f_\ii$ on $\G$:
\begin{enumerate}
\item 
$f_0$ is supported on $\B s\K_i$ and $f_0(bsx)=\a(b)$ for all $b\in\B$,
$x\in\K_i$.
\item
$f_\ii$ is supported on $\B\K_i$ and $f_\ii(bx)=\a(b)$ for all $b\in\B$,
$x\in\K_i$.
\end{enumerate}
As $\a$ is trivial on $\B\cap\K_i$,
these functions $f_0,f_\ii$ are well defined. 
They are in $\V$ but not in $\W$. 

\begin{lemm}
Given $f\in\V$, there is a unique function $w(f)\in\W$ such that: 
\begin{equation*}
f=f(s)f_0+f(1)f_\ii+w(f).
\end{equation*}
This defines a projection $w:\V\to\W$ with kernel spanned by $f_0$ and $f_\ii$. 
\end{lemm}

\begin{proof}
This follows from the fact that $s$ does not belong to $\B\K_i$.
\end{proof}

Let $\l$ be an $\H$-invariant linear form on $\V$.
It is characterized by $\l(f_0)$, $\l(f_\ii)\in\R$ 
and its restriction to $\W$.
As this restriction is $\H$-invariant, 
it is of the form $c\mu$ for a unique scalar $c\in\R$. 

\begin{coro}
The space $\V^{*\H}$ of $\H$-invariant linear forms on $\V$ has dimension $\<3$. 
\end{coro}

Now let $\l$ be a linear form on $\V$ extending $\mu$.  
We search for a necessary and sufficient condi\-tion on $\l(f_0)$, 
$\l(f_\ii)\in\R$ for $\l$ to be $\H$-invariant. 
By definition, this linear form is $\H$-invariant if and only if: 
\begin{equation*}
\l\left(\begin{pmatrix}x&0\\0&1\end{pmatrix}\cdot f\right)=\l(f)
\end{equation*}
for all $x\in\mult\F$ and $f\in\V$, and it is enough to check this condition 
for all $x$ of valuation $1$ and $f=f_0,f_\ii$. 
Let $t\in\mult\F$ be of valuation $1$.
We have:
\begin{eqnarray*}
\begin{pmatrix}t&0\\0&1\end{pmatrix}\cdot f_0
&=&\a_2(t)f_0+w\left(\begin{pmatrix}t&0\\0&1\end{pmatrix}\cdot f_0\right), \\
\begin{pmatrix}t&0\\0&1\end{pmatrix}\cdot f_\ii
&=&\a_1(t)f_\ii+w\left(\begin{pmatrix}t&0\\0&1\end{pmatrix}\cdot f_\ii\right).
\end{eqnarray*}
Thus the condition writes: 
\begin{equation*}
(1-\a_2(t))\l(f_0)=\mu_0(t) 
\quad\text{and}\quad
(1-\a_1(t))\l(f_\ii)=\mu_\ii(t)
\end{equation*}
for all $t\in\mult\F$ of valuation $1$, where:
\begin{equation*}
\mu_0(t)=\mu\left(w\left(\begin{pmatrix}t&0\\0&1\end{pmatrix}
\cdot f_0\right)\right)
\quad\text{and}\quad
\mu_\ii(t)=\mu\left(w\left(\begin{pmatrix}t&0\\0&1\end{pmatrix}
\cdot f_\ii\right)\right). 
\end{equation*}

\begin{lemm}
We have: 
\begin{equation*}
\mu_0(t)=-\a_2(t)^{1-i}\int\limits_{\mult\o}\a_2(x)^{-1}\ dx 
\quad\text{and}\quad
\mu_\ii(t)=\a_1(-1)\a_1(t)^i\int\limits_{\mult\o}\a_1(x)^{-1}\ dx.
\end{equation*}
\end{lemm}

\begin{proof}
Given $x\in\mult\F$, write $m\in\ZZ$ for the valuation of $x$ (normalized in 
such a way that any uniformizer has valuation $1$) and: 
\begin{equation*}
\iota(x)=\begin{pmatrix}1&x\\0&1\end{pmatrix}. 
\end{equation*}
We have: 
\begin{equation}
\label{E1}
s\iota(x)=s\begin{pmatrix}1&x\\0&1\end{pmatrix}\in\B s\K_i
\quad\Leftrightarrow\quad m\>i
\end{equation}
and: 
\begin{equation}
\label{E2}
s\iota(x)
=\begin{pmatrix}-x^{-1}&1\\0&x\end{pmatrix}
\begin{pmatrix}1&0\\x^{-1}&1\end{pmatrix}
\in\B\K_i
\quad\Leftrightarrow\quad m\<-i.
\end{equation}
Note that:
\begin{equation*}
s\iota(x)\begin{pmatrix}t&0\\0&1\end{pmatrix}=
s\begin{pmatrix}1&x\\0&1\end{pmatrix}
\begin{pmatrix}t&0\\0&1\end{pmatrix}
=s\begin{pmatrix}t&0\\0&1\end{pmatrix}
\begin{pmatrix}1&xt^{-1}\\0&1\end{pmatrix}
=\begin{pmatrix}1&0\\0&t\end{pmatrix}s\iota(xt^{-1}).
\end{equation*}
We have:
\begin{eqnarray*}
\mu_0(t)&=&
\int\limits_{\mult\F}
\left[f_0\left(s\iota(x)\begin{pmatrix}t&0\\0&1\end{pmatrix}\right)
-\a_2(t) f_0\left(s\iota(x)\right)\right]\a_2(x)^{-1}\ dx,\\
\mu_\ii(t)&=&
\int\limits_{\mult\F}
\left[f_\ii\left(s\iota(x)\begin{pmatrix}t&0\\0&1\end{pmatrix}\right)
-\a_1(t) f_\ii\left(s\iota(x)\right)\right]\a_2(x)^{-1}\ dx.
\end{eqnarray*}
Let $\phi_0(x,t)$ and $\phi_\ii(x,t)$ denote the functions into brackets in
the formulas above, respectively. 
We use formulas \eqref{E1} and \eqref{E2} above. 
For $\phi_0(x,t)$ we have the following:
\begin{enumerate}
\item 
if $m\>i+1$, then $\phi_0(x,t)=\a_2(t)-\a_2(t)=0$;
\item
if $m=i$, then $\phi_0(x,t)=-\a_2(t)$;
\item
if $m\<i-1$, then $\phi_0(x,t)=0$.
\end{enumerate}
For $\phi_\ii(x,t)$ we have:
\begin{enumerate}
\item 
if $m\>-i+2$, then $\phi_\ii(x,t)=0$;
\item
if $m=-i+1$, then $\phi_\ii(x,t)=\a_1(-tx^{-1})\a_2(x)$;
\item 
if $m\<-i$, then $\phi_\ii(x,t)=\a_1(-tx^{-1})\a_2(x)-\a_1(-tx^{-1})\a_2(x)=0$. 
\end{enumerate}
Therefore we have: 
\begin{equation*}
\mu_0(t)=-\a_2(t)\int\limits_{\mult\o}\a_2(t^ix)^{-1}\ dx 
\quad\text{and}\quad
\mu_\ii(t)=\int\limits_{\mult\o}\a_1(-t^{1-(1-i)}x^{-1})\ dx. 
\end{equation*}
This ends the proof of the lemma. 
\end{proof}

We now have the following result. 

\begin{theo}
The linear form $\mu$ can be extended to an $\H$-invariant linear form on $\V$ 
if and only if one of the two conditions below is satisfied: 
\begin{enumerate}
\item 
$q\neq1$ in $\R$ and $\a_1,\a_2$ are nontrivial.
\item 
$q=1$ in $\R$. 
\end{enumerate}
\end{theo}

\begin{proof}
If $\a_1,\a_2$ are ramified (that is, nontrivial on $\mult\o$), then:
\begin{equation*}
\int\limits_{\mult\o}\a_1(x)^{-1}\ dx=\int\limits_{\mult\o}\a_2(x)^{-1}\ dx=0. 
\end{equation*}
Thus $\mu$ can be extended uniquely to an $\H$-invariant linear form 
$\l$ on $\V$, by setting $\l(f_0)=\l(f_\ii)=0$. 
If $\a_i$ is unramified for some $i\in\{1,2\}$, then:
\begin{equation*}
\int\limits_{\mult\o}\a_i(x)^{-1}\ dx=q-1. 
\end{equation*}
Fix a uniformizer $\w$ of $\F$ and put $z_i=\a_i(\w)$. 
\begin{enumerate}
\item 
If $i=1$, the condition on $\l(f_\ii)$ writes: 
\begin{equation}
\label{3C1}
(1-z_1)\l(f_\ii)=z_1^{i}(q-1).
\end{equation}
If $z_1\neq1$, then \eqref{3C1} has a unique solution:
\begin{equation*}
\label{CNS1}
\l(f_\ii)=z_1^{i}\cdot\frac{q-1}{1-z_1}.
\end{equation*}
If $z_1=1$, then \eqref{3C1} has a solution if and only if we have $q=1$ in 
$\R$, and in that case any value of $\l(f_\ii)$ in $\R$ is a solution. 
\item 
If $i=2$, the condition on $\l(f_0)$ writes: 
\begin{equation}
\label{3C2}
(1-z_2)\l(f_0)=-z_2^{1-i}(q-1).
\end{equation}
If $z_2\neq1$, then \eqref{3C2} has a unique solution:
\begin{equation*}
\label{CNS2}
\l(f_0)=-z_2^{1-i}\cdot\frac{q-1}{1-z_2}. 
\end{equation*}
If $z_2=1$, then \eqref{3C2} has a solution if and only if we have $q=1$ in 
$\R$, and in that case any value of $\l(f_0)$ in $\R$ is a solution. 
\end{enumerate}
This ends the proof of the theorem. 
\end{proof}

Write $d(\V)$ for the dimension of $\V^{*\H}$ and $e(\V)$ for that of the 
subspace of $\H$-invariant linear forms which are trivial on $\W$. 

\begin{theo}\label{gl2mult}
Let $n$ denote the number of trivial characters among $\a_1,\a_2$. 
\begin{enumerate}
\item 
If $n=0$, then $d(\V)=1$ and $e(\V)=0$. 
\item 
If $n\>1$ and $q\neq1$ in $\R$, then $d(\V)=e(\V)=n$.
\item
If $n\>1$ and $q=1$ in $\R$, then $d(\V)=n+1$ and $e(\V)=n$.
\end{enumerate}
\end{theo}

\begin{proof}
If $q\neq1$ in $\R$, the result is as in the complex case. 
If $q=1$ in $\R$, then $\mu$ can always be extended to an $\H$-invariant 
linear form on $\V$, that is, we have an exact sequence: 
\begin{equation*}
0\to(\V/\W)^{*\H}\to\V^{*\H}\to\W^{*\H}\to0
\end{equation*}
of $\R$-vector spaces and the dimension of $\W^{*\H}$ is $1$. 
One easily checks that $e(\V)=n$. 
The result follows. 
\end{proof}

\subsection{The classification of $\GL_1(\F)$-distinguished irreducible 
  representations of $\GL_2(\F)$} 
\label{defqc}

For any irreducible (smooth) representation $\pi$ of $\G$, let $d(\pi)$ denote the 
dimension of its space of $\H$-invariant linear forms. 

Recall that $\qc$ denotes the quantum characteristic:
\begin{equation*}
\qc=
\left\{
\begin{array}{l}
0 \text{ if $\R$ has characteristic $0$},\\
\text{the smallest positive integer $k\>2$ such that 
$1+q+\dots+q^{k-1}=0$ in $\R$ otherwise}. 
\end{array}
\right.
\end{equation*}

An irreducible representation of $\G$ is said to be \textit{cuspidal} if it 
does not embed in any $\V(\a_1,\a_2)$ with $\a_1,\a_2\in\GH_1$.
Just as in the complex case, we have the following result for cuspidal 
represen\-tations.  

\begin{prop}
\label{distcusp}
All cuspidal irreducible representations $\pi$ of $\G$ are $\H$-distinguished, 
with $d(\pi)=1$.
\end{prop}

\begin{proof}
See Paragraph \ref{sectioncusp} for a proof, where we treat the more general case of
$\G_n$, $n\>2$.
\end{proof}

Now let $\St$ denote the Steinberg representation of $\G$, that is the unique 
nondegenerate irre\-ducible subquotient of $\V=\Ind^\G_\B(1\otimes1)$
(see \cite[III.1]{Vigb}). 

For the following lemma, see \cite[\S6]{MSc}.

\begin{lemm}
\label{einstein}
If $\qc=2$, then $\St\cdot\chi$ is cuspidal for all $\chi\in\GH_1$.
\end{lemm}

If $\qc=2$, then Proposition \ref{distcusp} implies that $\St\cdot\chi$ is 
$\H$-distinguished with $d(\St\cdot\chi)=1$ for all $\chi\in\GH_1$.
Assume now that $\qc\neq2$.
Thus $\V$ has length $2$ and we have an exact sequence: 
\begin{equation*}
0\to\chi\circ\det\to\V\cdot\chi=\Ind^\G_\B(\chi\otimes\chi)\to\St\cdot\chi\to0
\end{equation*}
of representations of $\G$. 
If $\chi$ is nontrivial, then any $\H$-invariant linear form on $\V\cdot\chi$ is 
trivial on $\chi\circ\det$. 
We thus have $d(\St\cdot\chi)=d(\V\cdot\chi)=1$. 
If $\chi=1$, we have:
\begin{equation*}
d(\St)\<d(\V)\<d(\St)+1.
\end{equation*}
As $\l_0$ and $\l_\ii$ are $\H$-invariant linear form on $\V$ which are 
nonzero on the subspace of constant functions, we get $d(\St)=d(\V)-1$. 
Finally, we have the following result. 

\begin{theo}\label{dimfor2}
\begin{enumerate}
\item 
An irreducible representation of $\G$ is $\H$-distinguished if and only if it 
is not a nontrivial $1$-dimensional representation. 
\item
Let $\pi$ be an $\H$-distinguished irreducible representation of $\G$. 
Then $d(\pi)\<2$, with equality if and only if $q=1$ in $\R$ and we are in one
of the following cases:
\begin{enumerate}
\item 
$\pi$ is the Steinberg representation $\St$ and $\R$ has characteristic $>2$~; 
\item 
$\pi$ is a principal series representation
$\V(1,\chi)=\Ind^\G_\B(1\otimes\chi)$ with $\chi\in\GH_1$ nontrivial.
\end{enumerate}
\end{enumerate}
\end{theo}

\section{General results on modulo $\ell$ representations of $\G_n$}\label{modlresults}

\subsection{More notation}

Let $\a=(n_1,\dots,n_r)$ be a composition of $n$, 
that is, a family of positive integer whose sum is $n$. 
We denote by $\M_{\alpha}$ the subgroup of $\G_n$ of invertible matrices
which are diagonal by blocks of size $n_1,\dots,n_r$ respectively (it is
isomorphic to $\G_{n_1}\times\dots\times\G_{n_r}$) and by $\P_{\a}$ the
subgroup of $\G_n$ generated by $\M_\a$ and the upper triangular matrices.  

\emph{We choose once and for all a square root of q in $\R$}.  
We write $\rp_\a$ for the normalized Jacquet func\-tor associated to $(\M_\a,\P_\a)$ and 
$\ip_\a$ for its right adjoint functor, that is, normalized parabolic in\-duc\-tion. 
If $\pi_1,\dots,\pi_r$ are smooth $\R$-representations of $\G_{n_1},\dots,\G_{n_r}$ 
respectively, we write: 
\begin{equation}
\label{indu}
\pi_1 \times \pi_2\times \dots \times \pi_r=
\ip_\a( \pi_1 \otimes \pi_2 \otimes \dots \otimes \pi_r).
\end{equation} 

Given a smooth representation $\pi$ of finite length, we write 
$\sy{\pi}$ for its semi-simplification and $\pi^\vee$ for its 
contragredient. 

We write $\nu$ for the normalized absolute value of $\F$, giving value 
$q^{-1}$ to any uniformizer. 
More generally, given integers $k\in\ZZ$ and $n\>1$, we write:
\begin{equation*}
\nu_{n}^{k/2}:g\mapsto(\sqrt{q})^{-k\cdot{\rm val}(\det(g))}
\end{equation*}
where $\sqrt{q}$ is the square root of $q$ in $\R$ that has been fixed above, 
${\rm val}$ is the normalized valuation on $\F$ and $\det$ is the determinant map from 
$\G_n$ to $\mult\F$. 

We also write $\1_n$ for the trivial character of $\G_n$, $n\>1$, and $1$ for $\1_1$.

\subsection{The Geometric Lemma}

We give here a combinatorial version of 
Bernstein-Zelevinski's Geometric Lemma \cite{BZ} 
(see also \cite[II.2.19]{Vigb}).  
Let $\a=(n_1,\dots,n_r)$ and $\b=(m_1,\dots,m_s)$ be two compositions of $n\>1$.
For each $i\in\{1,\ldots,r\}$, let $\pi_i\in\GH_{n_i}$. 
Let $\mathscr{B}^{\a,\b}$ be the set of all matrices $\B=(b_{i,j})$ whose 
coefficients are non-negative integers such that: 
\begin{equation*}
\label{2} 
\sum_{j=1}^s b_{i,j}=n_i, 
\quad i\in\{1,\ldots,r\}, \quad \sum_{i=1}^rb_{i,j}=m_j, \quad j\in\{1,\ldots,s\}.
\end{equation*}
Fix $\B\in\mathscr{B}^{\a,\b}$ and write $\a_i=(b_{i,1},\dots,b_{i,s})$ and 
$\b_j=(b_{1,j},\dots,b_{r,j})$ which are compositions of $n_i$ and  $m_j$ respectively. 
For all $i\in\{1,\ldots,r\}$, the semi-simplification of $\rp_{\a_i}(\pi_i)$ writes:
\begin{equation*}
\sy{\rp_{\a_i}(\pi_i)}=\sum\limits_{k=1}^{r_i}
\s^{(k)}_{i,1}\otimes\cdots\otimes\s^{(k)}_{i,s}, 
\quad
\s^{(k)}_{i,j}\in\GH_{b_{i,j}},
\quad r_i\>1.
\end{equation*}
For all $j\in\{1,\ldots,s\}$ and all $r$-tuples
$\boldsymbol{k}=(k_1,\ldots,k_r)$ with $1\<k_i\<r_i$, 
we write:
\begin{equation*}
\s_{j}^{(\boldsymbol{k})}=\s^{(k_1)}_{1,j}\times\cdots\times\s^{(k_r)}_{r,j},
\end{equation*}
which is a representation of $\G_{m_j}$.
Then we have: 
\begin{equation*}
\label{Druon}
[\rp_{\b}(\pi_1\times\dots\times\pi_r)]=
\sum_{\B}\sum_{\boldsymbol{k}}\
\s_{1}^{(\boldsymbol{k})}\otimes\cdots\otimes\s_{s}^{(\boldsymbol{k})}
\end{equation*}
in the Grothendieck group of finite length representations of $\M_{\b}$.

\subsection{Cuspidal support}

An irreducible representation of $\G_n$ with $n\>1$ is said to be 
\textit{cuspidal} if it does not embed in any representation of the form 
\eqref{indu} with $r>1$.

By \cite[Theorem 2.1]{MSc}, for any irreducible representation 
$\pi\in\GH_n$ with $n\>1$, there are positive integers $n_1,\dots,n_r$
and cuspidal irreducible representations $\rho_i\in\GH_{n_i}$ with 
$i\in\{1,\dots,r\}$ such that $n=n_1+\dots+n_r$ and $\pi$ embeds 
in $\rho_1\times\dots\times\rho_r$.
Moreover, there is a permutation $w$ of 
the set $\{1,2,\dots,r\}$ such that $\pi$ is a quotient 
of $\rho_{w(1)}\times\dots\times\rho_{w(r)}$.

The family $(\rho_1,\dots,\rho_r)$, which depends on the choice of 
$\sqrt{q}$, is unique up to permutation.
Its class up to permutation, denoted $[\rho_1]+\dots+[\rho_r]$, 
is called the \textit{cuspidal support} of $\pi$.

\begin{prop}[\cite{MSc}, Proposition 5.9]
\label{Lanzmann}
Let $\pi,\s$ be irreducible representations of $\G_n,\G_m$ respectively. 
Write $[\pi_1]+\dots+[\pi_r]$ and $[\s_1]+\dots+[\s_s]$ for the cuspidal supports 
of $\pi$ and $\s$ respectively.
Assume that for all $i\in\{1,\dots,r\}$, $j\in\{1,\dots,s\}$ and $k\in\ZZ$, 
the cuspidal irreducible 
representations $\pi_i\cdot\nu^k$ and $\s_j$ are not isomorphic.
Then $\pi\times\s$ is irreducible. 
\end{prop}

\subsection{Three lemmas about irreducibility}

The following lemma is a particular case of \cite[Lemma 6.1]{MSu}, 
which will be of crucial importance to us. 
Recall that $\ee$ is the order (possibly infinite) of $q$ in $\mult\R$.

\begin{lemm}
\label{UIQ}
Assume that $\ee>1$.
Let $n\>2$, and let $\rho\in\GH_{n-1}$, $\chi\in\GH_1$.
Then the represen\-ta\-tion $\pi=\rho\times\chi$ 
possesses a unique irreducible quotient, denoted $\UIQ(\pi)$, 
and a unique irreducible sub\-repre\-sen\-tation, denoted $\UIS(\pi)$.
\end{lemm}

Note that, by passing to the contragredient, we have:
\begin{equation*}
\label{Qcontra}
\UIQ(\rho\times\chi)^\vee=\UIS(\rho^\vee\times\chi^{-1}),
\quad
\UIS(\rho\times\chi)^\vee=\UIQ(\rho^\vee\times\chi^{-1}).
\end{equation*}
There is also a similar result for $\chi\times\rho$, and we have:
\begin{equation*}
\label{Qswap}
\UIQ(\chi\times\rho)=\UIS(\rho\times\chi),
\quad
\UIS(\chi\times\rho)=\UIQ(\rho\times\chi).
\end{equation*}
From this lemma we deduce the following example. 

\begin{exem}
\label{poutargue}
Assume that $n\>2$, and write:
\begin{equation}
\label{DEFVN}
\V_n = \nu_{n-1}^{1/2}\times\nu^{(n+1)/2}.
\end{equation}
If $e>1$, Lemma \ref{UIQ} implies that
$\V_n$ has a unique irreducible quotient, denoted $\LL_{n}$.
We write:
\begin{equation}
\LL_n = \UIQ(\V_n) = \UIQ(\nu_{n-1}^{1/2}\times\nu^{(n+1)/2}),
\quad
\text{for $e>1$}.
\end{equation}
When $e$ divides $n$, then $\LL_n$ is the trivial character 
(see Proposition \ref{P61}).
By taking the contragredient, $\LL_n^\vee$ is the unique irreducible 
subrepresentation of $\nu_{n-1}^{-1/2}\times\nu^{-(n+1)/2}$.
\end{exem}

The following irreducibility criterion 
will also be very useful to us.

\begin{lemm}[\cite{MSc}, Lemme 2.5]
Let $\pi$ be a smooth representation of $\G_{n}$. 
Assume that there are two irreducible representations $\s\in\GH_{a}$ and 
$\tau\in\GH_{b}$ with $a,b\>1$ and $a+b=n$, such that: 
\begin{enumerate}
\item
$\pi$ is a subrepresentation of $\s\times\tau$ and a quotient of 
$\tau\times\s$; 
\item
the multiplicity of $\s\otimes\tau$ in $\rp_{(a,b)}(\s\times\tau)$ is $1$.
\end{enumerate}
Then the representation $\pi$ is irreducible. 
\end{lemm}

Finally, we will use the following lemma
(which follows from \cite[Proposition 2.2]{MSc}).

\begin{lemm}
\label{PERM}
Assume the induced representation \eqref{indu} is irreducible. 
Then for all permuta\-tion $w$ of $\{1,2,\dots,n\}$ there is an 
isomorphism 
$\pi_{w(1)}\times\pi_{w(2)}\tdt\pi_{w(r)}\simeq
\pi_1\times\pi_2\tdt\pi_r$. 
\end{lemm}

\subsection{Classification of $\GH_n$ by multisegments}
\label{repmodl}

In \cite{MSc} M\'\i nguez and Sécherre give a classification of the union of 
all $\GH_{n}$'s in terms of multisegments, that generalizes \cite{Z,T,Vigs}. 
We will need some properties of this classification, that we recall below. 

Given two half-integers $a,b\in\frac{1}{2}\ZZ$, we write:
\begin{equation*}
a\equiv b \text{ if }
\left\{
\begin{array}{ll}
a-b\in\ee\ZZ
& \text{if $\R$ has positive characteristic},\\
a=b& \text{otherwise}.
\end{array}
\right.
\end{equation*}

\begin{defi}
\begin{enumerate}
\item
A \textit{segment} is a pair $(a,b)$ of half-integers such that $a\<b$. 
\item
Two segments $(a,b)$ and $(c,d)$ are \textit{equivalent} if $b-a=d-c $ and 
$a\equiv c$.
The equivalence class of $(a,b)$ will be denoted $[a,b]$
(and just $[a]$ if $b=a$).
\item
A \textit{multisegment} is a formal finite sum of classes of segments, that is 
a element in the free semigroup generated by classes of segments. 
\end{enumerate}
\end{defi}

Let $\l=(\l_1,\l_2,\dots)$ and $\mu=(\mu_1,\mu_2,\dots)$ be 
two partitions of a given integer $n$.
We say that $\l$ dominates $\mu$, denoted $\l\trianglerighteq\mu$, if: 
\begin{equation*}
\l_1+\dots+\l_k\>\mu_1+\dots+\mu_k
\end{equation*}
for all integers $k\>1$. 
We write $\l\triangleright\mu$ if we have in addition $\l\neq\mu$.

Given a nonzero multisegment $\m=[a_1,b_1]+\dots+[a_r,b_r]$, 
write $n_i=b_i-a_i+1$ for all integer $i\in\{1,\dots,r\}$
and let $\l(\m)$ denote the partition 
associated with $(n_1,n_2,\dots,n_r)$. 
The \textit{length} of $\m$ is the sum $n=n_1+n_2+\dots+n_r$. 

One of the main results of \cite{MSc} is the construction of a map 
$\m\mapsto\Z(\m)$ that associates to any multisegment $\m$ 
a class of irreducible representation $\Z(\m)$ with 
the following properties:

\medskip

\begin{tabular}{lp{14cm}}
{\bf P1} & 
If $\m$ is a segment $[a,b]$ of length $n\>1$, then $\Z([a,b])$
is the character $\nu_{n}^{a+(n-1)/2} \in \GH_{n}$. \\
\end{tabular}

\medskip

\begin{tabular}{lp{14cm}}
{\bf P2} &
If $\m=[a_1,b_1]+\dots+[a_r,b_r]$, then $\Z(\m)$ 
occurs as a subquotient of the rep\-re\-sen\-tation
$\Z([a_1,b_1])\tdt\Z([a_r,b_r])$ with multiplicity $1$. \\
\end{tabular}

\medskip

\begin{tabular}{lp{14cm}}
{\bf P3} &
If $\pi$ is an irreducible subquotient  of 
$\Z([a_1,b_1])\tdt\Z([a_r,b_r])$, then there exists a unique multi\-segment 
$\n=[c_1,d_1]+\dots+[c_s,d_s]$ such that $\pi=\Z(\n)$.
Moreover, we have $\l(\n)\trianglerighteq\l(\m)$ and:
\end{tabular}
\begin{equation*}
\sum\limits_{i=1}^{s} ([c_i]+\dots+[d_i])=
\sum\limits_{i=1}^{r} ([a_i]+\dots+[b_i]).
\end{equation*}

\medskip

\begin{tabular}{lp{14cm}}
{\bf P4} &
If $k$ is a half-integer, then
$\Z([a_1+k,b_1+k]+\dots+[a_r+k,b_r+k])=\Z(\m)\cdot\nu^k$. \\
\end{tabular}

\medskip

\begin{tabular}{lp{14cm}}
{\bf P5} & The contragredient of $\Z(\m)$ is $\Z(\m^\vee)$ with 
$\m^\vee=[-b_1,-a_1]+\dots+[-b_r,-a_r]$.\\
\end{tabular}

\medskip

We finally have the following definition and result. 

\begin{defi}
Two segments $[a,b]$ and $[c,d]$ are \textit{linked} if:
\begin{enumerate}
\item
the length of $[a,b]$ is greater than or equal to that of $[c,d]$, 
and there exists a half-integer $k$ such that $c\<k\<d$, 
and either $k\equiv b+1$ or $k\equiv a-1$;
\item
the length of $[c,d]$ is greater than or equal to that of $[a,b]$, 
and there exists a half-integer $k$ such that $a\<k\<b$, 
and either $k\equiv d+1$ or $k\equiv c-1$.
\end{enumerate}
\end{defi}

\begin{prop}[\cite{MSc}, Théorème 7.26]
\label{Linked}
Let $\Delta_1,\dots,\Delta_r$ be segments. 
The representation $\Z(\Delta_1)\times\dots\times\Z(\Delta_r)$ is ir\-reducible if and 
only if for all $i\neq j$, the segments $\Delta_i,\Delta_j$ are not linked. 
\end{prop}

\subsection{Product of two characters}
\label{product}

Here are some useful properties of the representation $\Z([a,b])$ for a 
segment $[a,b]$.

\begin{prop}
\label{Rappel}
Let $[a,b]$ be a segment of length $n\>2$, and let $k\in\{1,\dots,n-1\}$.
\begin{enumerate}
\item
We have $\rp_{(k,n-k)}(\Z([a,b]))=\Z([a,a+k-1])\otimes\Z([a+k,b])$.
\item
We have $\bar{\rp}_{(k,n-k)}(\Z([a,b]))=\Z([a+n-k,b])\otimes\Z([a,a+n-k-1])$
where $\bar{\rp}_{(k,n-k)}$ de\-no\-tes the Jacquet functor associated 
to $\M_{(k,n-k)}$ and the parabolic subgroup opposite to $\P_{(k,n-k)}$. 
\item
Assume $e>1$.
Then $\Z([a,b-1])\times\nu^b$
has a unique irreducible subrepresen\-tation and 
$\Z([a+1,b])\times\nu^a$
has a unique irreducible quotient, both isomorphic to $\Z([a,b])$. 
\end{enumerate}
\end{prop}

\begin{proof}
See \cite[Propositions 7.16 and 7.17]{MSc}.
\end{proof}

\begin{prop}
\label{P61}
Assume $e>1$.
Let $a\<b$ be integers and write $\pi(a,b)=\Z([a,b])\times1$. 
\begin{enumerate}
\item
If $a\nequiv1$ and $b\nequiv-1$, then $\pi(a,b)$ is irreducible. 
\item
If $a\equiv1$ and $b\nequiv-1$, then $\pi(a,b)$ has length $2$
and we have an exact sequence: 
\begin{equation*}
0\to\Z([a,b]+[0])\to\pi(a,b)\to\Z([a-1,b])\to0.
\end{equation*}
\item
If $a\nequiv1$ and $b\equiv-1$, then $\pi(a,b)$ has length $2$
and we have an exact sequence: 
\begin{equation*}
0\to\Z([a,b+1])\to\pi(a,b)\to\Z([a,b]+[0])\to0.
\end{equation*}
\item
If $a\equiv1$ and $b\equiv-1$, then $\pi(a,b)$ has length $3$ 
with irreducible subquotients $\Z([a-1,b])$, $\Z([a,b+1])$ and 
$\Z([a,b]+[0])$.
\end{enumerate}
\end{prop}

\begin{proof}
Case 1 follows from Proposition \ref{Linked}.
Moreover, the representation 
$\Z([a,b]+[0])$ always occurs as a subquotient with multiplicity $1$ 
and the other irreducible subquotients of $\pi(a,b)$ are of the form 
$\Z(\n)$ with $\l(\n)\triangleright(n-1,1)$, where $n=b-a+2$.
Therefore we have $\l(\n)=(n)$, which 
implies that $\n$ is a segment.
Moreover, such an $\n$ must be of the form $[a,b+1]$ with $b\equiv-1$ 
or $[a-1,b]$ with $a\equiv1$.

Assume that $a\nequiv1$ and $b\equiv-1$. 
By the geometric lemma, the Jacquet module $\rp_{(n-1,1)}(\pi(a,b))$ 
is made of the subquotients $\Z([a,b])\otimes1$ and 
$\pi(a,b-1)\otimes\nu^{b}$, and both are irreducible. 
Thus the representation $\pi(a,b)$ has length $\<2$.
But Pro\-po\-sition \ref{Rappel} shows that $\Z([a,b+1])$ occurs as 
a subrepresentation of $\pi(a,b)$.
The result follows. 

The case where $a\equiv1$ and $b\nequiv-1$ is treated in a similar way.  
Thus it remains to study the case where $a\equiv1$ and $b\equiv-1$. 
In this case, $\pi(a,b-1)$ has length $2$, thus $\pi(a,b)$ has length $\<3$.
By Pro\-po\-sition \ref{Rappel} we see that the length is actually $3$ 
and we get the expected result.
\end{proof}

\begin{exem}
\label{poutargue2}
Assume that $n\>2$ and write:
\begin{equation}
\label{DEFPIN}
\MM_n = 
\Z\(\left[-\frac{n-3}{2},\frac{n-1}{2}\right]+\left[\frac{n+1}{2}\right]\).
\end{equation}
Assume $\ee>1$ and look at Example \ref{poutargue} for the definition of 
$\LL_n$. 
Then:
\begin{equation}
\label{DEFLaN}
\LL_{n}=
\left\{
\begin{array}{ll}
\MM_n 
& \text{if $e$ does not divide $n$}, \\
\1_n 
& \text{if $e$ divides $n$}.
\end{array}
\right.
\end{equation}
We also get:
\begin{equation*}
\LL_{n}^\vee=
\left\{
\begin{array}{ll}
\MM_n^\vee=\Z([-\frac{n-1}{2},\frac{n-3}{2}]+[-\frac{n+1}{2}])
& \text{if $e$ does not divide $n$}, \\
\1_n
& \text{if $e$ divides $n$}.
\end{array}
\right.
\end{equation*}
\end{exem}

If we want to go further, we need more properties of the representation 
$\Z(\m)$ for a multisegment $\m$.
Given $\chi_1,\chi_2\in\GH_1$, we write $\St(\chi_1,\chi_2)$ for the 
unique nondegenerate irreducible subquotient of  
$\chi_1\times\chi_2$
(see \cite[\S8]{MSc}).
If $\St_2$ is the Steinberg representation of $\G_2$ as in Paragraph 
\ref{defqc}, then: 
\begin{equation*}
\St(\chi_1,\chi_2) = 
\left\{
\begin{array}{ll}
\chi_1\times\chi_2 & \text{if $\chi_1\times\chi_2$ is irreducible}, \\
\St_2\cdot\chi_1\nu^{1/2} & \text{if $\chi_2=\chi_1\nu$},\\
\St_2\cdot\chi_2\nu^{1/2} & \text{if $\chi_1=\chi_2\nu$}.
\end{array}
\right.
\end{equation*}
Note that we have $\St(\chi_2,\chi_1)=\St(\chi_1,\chi_2)$. 
The following proposition follows from \cite[\S3.3.2]{MSu}.

\begin{prop}
\label{JacquetMu}
Let $\m$ be a multisegment of length $n$ and of the form $[a,b]+[c,d]$.  
Assu\-me that $b-a\>d-c$. 
Write $k=d-c+1$ and:
\begin{eqnarray*}
\upmu(\m) &=& (1,\dots,1,2,\dots,2) 
\text{ with $1$ occurring $n-2k$ times and $2$ occurring $k$ times},\\
\St(\m) &=& \nu^{a}\otimes\dots\otimes\nu^{a+n-2k-1}\otimes
\St(\nu^{a+n-2k},\nu^c) \otimes\dots\otimes\St(\nu^b,\nu^d).
\end{eqnarray*}
Then $\Z(\m)$ has the following property:

\medskip

\begin{tabular}{lp{14cm}}
{\bf P6} &
$\Z(\m)$ is the unique irreducible subquotient of 
$\Z([a,b])\times\Z([c,d])$ whose Jacquet module with respect to 
$\rp_{\upmu(\m)}$ contains $\St(\m)$ as a subquotient. \\
\end{tabular}
\end{prop}

\begin{prop}\label{length2seg}
Let $a,b\in\ZZ$ with $a\<b$ and write $\pi(a,b)=\Z([a,b])\times\Z([0,1])$. 
Assume that $e>1$. 
\begin{enumerate}
\item
$\Z([a,b]+[0,1])$ occurs as a subquotient of $\pi(a,b)$ with multiplicity $1$. 
\item
If $b\equiv0$, then $\Z([a,b+1]+[0])$ occurs as a subquotient of $\pi(a,b)$ 
with multiplicity $1$.  
\item
If $a\equiv1$, then $\Z([a-1,b]+[1])$ occurs as a subquotient of $\pi(a,b)$ 
with multiplicity $1$.  
\item
If $b\equiv-1$, then $\Z([a,b+2])$ occurs as a subquotient of $\pi(a,b)$. 
\item
If $a\equiv2$, then $\Z([a-2,b])$ occurs as a subquotient  of $\pi(a,b)$.
\item
If $b\equiv0$ and $a\equiv1$, then $\Z([a-1,b+1])$ occurs as a subquotient 
of $\pi(a,b)$. 
\end{enumerate}
Any irreducible subquotient of $\pi(a,b)$ is one of the representations 
in Cases 1 to 6.
Moreover, if $e>2$, the multiplicities in Cases 4, 5 and 6 are equal to $1$. 
\end{prop}

\begin{proof}
Case 1 follows from {\bf P2}.
Write $n=b-a+3$.
The other irreducible subquotients of $\pi(a,b)$ are of the form 
$\Z(\n)$ with $\l(\n)\triangleright(n-2,2)$.
Thus we have $\l(\n)=(n-1,1)$ or $\l(\n)=(n)$.
If $b\equiv 0$ (resp. $a\equiv 1$) then $\Z([a,b+1])\times\1$ 
(resp. $\Z([a-1,b])\times\nu$) is a subquotient of $\pi(a,b)$ by 
\cite[Lemme 7.34]{MSc}. 
It follows from Proposition \ref{P61} that the representations 
in Cases 2, 3 and 6 occur in $\pi(a,b)$. 
Cases 4 and 5 are treated similarly. 
We now show that these are the only possible subquotients of $\pi(a,b)$ and 
they appear with the specified multiplicity.

Assume first that $\n=[c,d]+[h]$ with $d-c+1=n-1$.
Then:
\begin{eqnarray*}
\upmu(\n) &=& (1,\dots,1,2),\\
\St(\n) &=& \nu^c\otimes\nu^{c+1}\otimes\dots
\otimes\nu^{d-1}\otimes\St(\nu^d,\nu^h).
\end{eqnarray*}
By using the geometric lemma, the semi-simplification of 
$\rp_{(n-2,2)}(\pi(a,b))$ is equal to:
\begin{equation*}
\Z([a,b])\otimes\Z([0,1])+
(\Z([a,b-1])\times\1)\otimes(\nu^b\times\nu)+
\pi(a,b-2)\otimes\Z([b-1,b]).
\end{equation*}
If $\Z(\n)$ occurs as a subquotient of $\pi(a,b)$, 
then $\St(\n)$ occurs in 
$\rp_{\upmu(\n)}((\Z([a,b-1])\times\1)\otimes(\nu^b\times\nu))$
and $\St(\nu^d,\nu^h)$ occurs in $\nu^b\times\nu$ with multiplicity $1$, 
which implies that $\St(\nu^d,\nu^h)=\St(\nu^b,\nu)$ and that $\Z(\n)$ occurs 
in $\pi(a,b)$ with multiplicity $1$.
By the geometric lemma, we get: 
\begin{equation*}
\rp_{(1,\dots,1)}(\Z([a,b-1])\times\1) = 
\sum\limits_{k=0}^{n-3} \nu^{a}\otimes\dots\otimes\nu^{a+k-1}\otimes
1\otimes\nu^{a+k}\otimes\dots\otimes\nu^{b-1}.
\end{equation*}
Thus there is a $k\in\{0,\dots,n\}$ such that:
\begin{equation*}
\nu^c\otimes\nu^{c+1}\otimes\dots\otimes\nu^{d-1} = 
\nu^{a}\otimes\dots\otimes\nu^{a+k-1}\otimes
1\otimes\nu^{a+k}\otimes\dots\otimes\nu^{b-1}.
\end{equation*}
Since $e>1$, comparing the exponents in the left hand side and the right hand 
side shows that $k$ must be either $0$ or $n-3=b-a$.
If $k=0$, then:
\begin{equation*}
\nu^c\otimes\nu^{c+1}\otimes\dots\otimes\nu^{d-1} = 
1\otimes\nu^{a}\otimes\dots\otimes\nu^{b-1}.
\end{equation*}
Thus we have $c\equiv0$, $a\equiv1$, $d\equiv b$ and $h\equiv1$.
It follows that $\n=[a-1,b]+[1]$.
If $k=n-3$, then:
\begin{equation*}
\nu^c\otimes\nu^{c+1}\otimes\dots\otimes\nu^{d-1} = 
\nu^{a}\otimes\dots\otimes\nu^{b-1}\otimes1.
\end{equation*}
Thus we have $c\equiv a$, $b\equiv0$, $d\equiv1$ and $h\equiv0$.
It follows that $\n=[a,b+1]+[0]$.

Assume now that $\n=[c,d]$ is a segment.
Thus:
\begin{eqnarray*}
\upmu(\n) &=& (1,\dots,1),\\
\St(\n) &=& \nu^c\otimes\nu^{c+1}\otimes\dots\otimes\nu^{d}.
\end{eqnarray*}
By using the geometric lemma, we get:
\begin{equation*}
\label{Jacq}
\rp_{\upmu(\n)}(\pi(a,b)) = 
\sum\limits_{0\<r\<s\<n} \nu^a\otimes\dots\otimes\nu^{a+r-1}\otimes
1\otimes\nu^{a+r}\otimes\dots\otimes\nu^{a+s-1}\otimes\nu\otimes
\nu^{a+s}\otimes\dots\otimes\nu^{b}.
\end{equation*}
If $\Z(\n)$ occurs as a subquotient of $\pi(a,b)$, there are integers $r\<s$ 
in $\{0,\dots,n\}$ such that:
\begin{equation*}
\nu^c\otimes\dots\otimes\nu^{d} = 
\nu^a\otimes\dots\otimes\nu^{a+r-1}\otimes
1\otimes\nu^{a+r}\otimes\dots\otimes\nu^{a+s-1}\otimes\nu\otimes
\nu^{a+s}\otimes\dots\otimes\nu^{b}.
\end{equation*}
If $e>2$, comparing the exponents in the left hand side and the right hand 
side shows that the only possible values for $r,s$ are:
\begin{enumerate}
\item $r=s=0$ (thus $a\equiv2$);
\item $r=s=n$ (thus $b\equiv-1$);
\item $r=0$ and $s=n$ (thus $a\equiv1$ and $b\equiv0$).
\end{enumerate}
In all these cases, $\St(\n)$ occurs with multiplicity $1$. 

If $e=2$, there are more possible values for $r,s$ (the condition is that
$s-r$ is even) and $\St(\n)$ may occur with multiplicity 
greater than $1$. 
\end{proof}

\subsection{Derivatives}

By \cite[III.1]{Vigb}, there is a theory of derivatives for 
mod $\ell$ representations of $\G_n$, $n\>1$ just as in the complex case.
Given a smooth representation $\pi$ of $\G_n$, $n\>1$ and an integer 
$k\in\{0,\dots,n\}$, 
we will write $\pi^{(k)}$ for its $k^{{\rm th}}$ deriva\-tive, 
which is a smooth representation of $\G_{n-k}$
(where $\G_0$ stands for the trivial group in the case $k=n$.)

The $k^{{\rm th}}$ deriva\-tive functor is exact from the category 
of smooth $\ell$-modular representations of $\G_n$ to that of smooth 
$\ell$-modular representations of $\G_{n-k}$, for all 
$k\in\{0,\dots,n\}$.
It is compatible with twisting by a character, that is, we have 
$(\pi\cdot\chi)^{(k)}=\pi^{(k)}\cdot\chi$ for any representation 
$\pi$ of $\G_n$, any character $\chi\in\GH_1$ and any $k\in\{0,\dots,n\}$. 

Recall that $[\pi]$ denotes the semi-simplification of a finite length 
representation $\pi$. 

\begin{lemm}
\label{dercal}
\begin{enumerate}
\item
Given a cuspidal irreducible representation $\rho$ of $\G_n$, 
its $k^{{\rm th}}$ deriva\-tive is zero for all $k\in\{1,\dots,n-1\}$, 
and we have $\rho^{(n)}=1$ for $k=n$.
\item 
Given a segment $[a,b]$, the first derivative of $\Z([a,b])$ is $\Z([a,b-1])$, 
and its $k^{{\rm th}}$ deriva\-tive is zero for all $k\in\{2,\dots,n\}$.
\item 
Let $\pi,\s$ be finite length representations of $\G_n,\G_m$ respectively, 
with $m\>n\>1$.
Then~:
\begin{equation*}
[(\pi\times\s)^{(k)}] = [\pi\times\s^{(k)}] + [\pi^{(1)}\times\s^{(k-1)}] + \dots
+ [\pi^{(i)}\times\s^{(k-i)}]
\end{equation*}
for all $k\in\{0,\dots,n+m\}$, where $i={\rm min}(n,k)$.
\end{enumerate} 
\end{lemm}

\begin{proof}
Points (1) and (2) follows from V.9.1 (a) and (b) in \cite{Vigs}. 
For (3), see \cite[III.1.10]{Vigb}.  
\end{proof}

\section{On the $e=1$ case}
\label{Castor}

In this section, we assume that $e=1$ and $n\>2$.
Write $\K_n=\GL_n(\o)$ and let $\K_n(1)$ be the normal 
subgroup of $\K_n$ made of all matrices that are congruent to 
$1$ mod $\p$.
Both are compact open subgroups of $\G_n$, and the quotient 
$\K_n/\K_n(1)$ is canonically isomorphic to the finite group $\GL_n(q)$ 
of $n\times n$ invertible matrices with entries in the residue field of $\Oo$. 

Given a smooth representation $(\pi,\W)$ of $\G_n$, 
write $\overline{\W}$ for the space of $\K_n(1)$-fixed vectors of $\W$ 
and write $\overline{\pi}$ for the representation of $\GL_n(q)$ on 
$\overline{\W}$. 

This defines an exact functor
from the category of smooth $\R$-representations of $\G_n$ to 
that of $\R$-representations of $\GL_n(q)$.

We have defined two representations $\V_n$ and $\Pi_n$
in \eqref{DEFVN} and \eqref{DEFPIN}.
Note that 
$\V_n=\Cc^\infty_{c}(\X,\R)$ with $\X=\P_{(n-1,1)}\backslash\G_n$.
Its contains $\Pi_n$ as a subquotient with mul\-ti\-pli\-city one, 
$\1_n$ with some multiplicity and no other irreducible 
subquotient.
It is a self\-dual represen\-tation of $\G_n$.

Thanks to the Iwasawa decomposition $\G_n=\P_{(n-1,1)}\K_n$, 
the restriction of $\V_n$ to $\K_n$ is $\W_n=\Cc_{c}^\infty(\Y,\R)$ 
with $\Y=(\K_n\cap\P_{(n-1,1)})\backslash\K_n$.
Therefore we have:
\begin{equation*}
\overline{\V}_n = \Cc^\infty_{c}(\Y/\K_n(1),\R),
\end{equation*}
which identifies with the space of $\R$-valued functions on 
$\overline{\X}=\P_{(n-1,1)}(q)\backslash\GL_n(q)$, where we write
$\P_{(n-1,1)}(q)$ for the 
standard maximal parabolic subgroup of $\GL_n(q)$ corresponding to $(n-1,1)$.

\begin{lemm}
\label{Api}
For $n\>2$, there exists a unique irreducible representation $\pi_n$ of $\GL_n(q)$ 
having the following properties: 
\begin{enumerate}
\item 
If $\ell$ does not divide $n$, then $\overline{\V}_n$ is
semisimple of length $2$, 
with irreducible sub\-quotients $\bar{\1}_n$ and $\pi_n$.
\item 
If $\ell$ divides $n$, then $\overline{\V}_n$ is indecomposable
of length $3$, with irre\-ducible subquotients $\bar{\1}_n$ (with multiplicity 
$2$) and $\pi_n$.
\end{enumerate}
\end{lemm}

\begin{proof}
Note that $\bar{\1}_n$ occurs as a subrepresentation of $\overline{\V}_n$ 
(the space of $\R$-valued constant func\-tions on $\overline{\X}$). 
Write $\psi$ for the $\GL_n(q)$-invariant linear form on $\overline{\V}_n$ 
that associates to a function the sum of its values on $\overline{\X}$.  
The set $\overline{\X}$ has cardinality: 
\begin{equation*}
(\GL_n(q):\P_{(n-1,1)}(q)) = \frac{q^n-1}{q-1} = 1+q+\dots+q^{n-1} 
\end{equation*}
which is $0$ in $\R$ if and only if $\ell$ divides $n$.
Thus the constant func\-tions belong to the kernel of $\psi$ if and only if 
$\ell$ divides $n$.
According to \cite[11]{Jamesb}, we have the following properties: 
\begin{enumerate}
\item
The kernel $\SS_n$ of $\psi$ 
(denoted $\SS_{(n-1,1)}$ in \cite{Jamesb}, whereas $\overline{\V}_n$ is 
denoted $\M_{(n-1,1)}$) has a unique irreducible quotient 
$\pi_n$.
\item
The semi-simplification of $\overline{\V}_n$ contains $\pi_n$ with multiplicity $1$ 
and $\bar{\1}_n$ with some multiplicity $\>1$, and no other irreducible 
subquotient. 
\end{enumerate}
By \cite[20.7]{Jamesb}, the multiplicity of $\bar{\1}_n$ in $\overline{\V}_n$ 
is $1$ if $\ell$ does not divide $n$, and $2$ otherwise. 
It remains to prove that $\overline{\V}_n$ has the expected structure. 

We first assume that $\ell$ does not divide $n$.
Since $\bar{\1}_n$ occurs as a subrepresentation of $\overline{\V}_n$, 
$\pi_n$ must be a quotient of $\overline{\V}_n$.
Since $\overline{\V}_n$ is selfdual, it follows that $\pi_n$ is selfdual, thus 
it also occurs as a subrepresentation of $\overline{\V}_n$.
We thus have two nonzero maps $\pi_n\to\overline{\V}_n$ and 
$\overline{\V}_n\to\pi_n$, whose
composite is nonzero (or else it would contradict the fact that $\pi_n$ occurs 
with multiplicity $1$).
Therefore $\overline{\V}_n$ is semisimple.

Assume now that $\ell$ divides $n$.
By \cite{Jamesb}, the representation $\SS_n$ is indecomposable (it has length
$2$ and a unique irreducible quotient). 
Since $\overline{\V}_n$ is selfdual, it implies that $\overline{\V}_n$ is indecomposable.  
\end{proof}

\begin{prop}
\label{ssV1}
\begin{enumerate}
\item
The representation $\overline{\Pi}_n$ is irreducible and isomorphic to $\pi_n$.
\item 
If $\ell$ does not divide $n$, the representation $\V_n$ is 
semisimple of length $2$. 
\item 
If $\ell$ divides $n$, the representation $\V_n$ is indecomposable of
length $3$,
with irre\-ducible subquotients $\1_n$ (with multiplicity $2$) 
and $\Pi_n$.
\end{enumerate}
\end{prop}

\begin{proof}
By \cite{Vigb}, II.5.8 and II.5.12, all irreducible subquotients of $\V_n$ 
have level $0$, thus they are not killed by the functor
$\pi\mapsto\overline{\pi}$.  

We first assume that $\ell$ does not divide $n$.
By Lemma \ref{Api}, the representation $\V_n$ has length $2$, with irreducible 
sub\-quotients $\Pi_n$ and $\1_n$, thus $\overline{\Pi}_n$ must be irreducible
and isomorphic to $\pi_n$.
The same argument as in the proof of Lemma \ref{Api} shows that $\V_n$ is 
semisimple. 

Assume now that $\ell$ divides $n$.
By Lemma \ref{Api} the representation $\V_n$ has length $\<3$.
Assume it has length $2$. 
Then the argument of the proof of Lemma \ref{Api} 
implies that $\V_n=\1_n\oplus\Pi_n$.
Thus the one-dimensional space $\Hom_{\G_n}(\V_n,\1_n)$ is generated by a 
linear form $\l$ which is nonzero on the subspace of constant 
functions.  
Since $\K_n(1)$ is a pro-$p$-group,
$\K_n(1)$-invariant and $\K_n(1)$-co\-invariant vectors of $\V_n$ are 
canonically identified.
The $\K_n$-invariant linear form 
$\l$ thus induces a $\GL_n(q)$-invariant linear form on $\overline{\V}_n$, 
which is equal to $\psi$ upto a nonzero scalar. 
But $\psi$ is zero on constant functions, which contradicts the fact that $\l$ 
is nonzero. 
This gives us a contradiction, and thus $\V_n$ has length $3$. 
Now since $\overline{\V}_n$ is indecomposable, it follows that $\V_n$ is 
indecomposable.
We also get that $\overline{\Pi}_n$ must be irreducible
and isomorphic to $\pi_n$.
\end{proof}

\begin{defi}
\label{poutargue3}
Assume $e=1$ and let $n\>2$.
In parallel with Example \ref{poutargue2}, we define:
\begin{equation*}
\LL_n=
\left\{
\begin{array}{ll}
\Pi_n
& \text{if $\ell$ does not divide $n$}, \\
\1_n & \text{if $\ell$ divides $n$}.
\end{array}
\right.
\end{equation*}
\end{defi}

In conclusion, if we summarize Example \ref{poutargue2} and
Definition \ref{poutargue3}, we get the following definition 
of $\LL_n$. 

\begin{defi}
\label{poutargue4}
Assume $e$ is arbitrary, 
and recall that $\qc$ is the quantum characteristic (see Paragraph \ref{defqc}). 
For $n\>2$, we define:
\begin{equation}
\label{DEFLAN}
\LL_n=
\left\{
\begin{array}{ll}
\MM_n=\Z([-\frac{n-3}{2},\frac{n-1}{2}]+[\frac{n+1}{2}])
& \text{if $\qc$ does not divide $n$}, \\
\1_n & \text{if $\qc$ divides $n$}.
\end{array}
\right.
\end{equation}
\end{defi}

Thanks to Example \ref{poutargue2}, note that we also have:
\begin{equation}
\label{CONTRAL}
\LL_n^\vee=
\left\{
\begin{array}{ll}
\MM_n^\vee=\Z([-\frac{n-1}{2},\frac{n-3}{2}]+[-\frac{n+1}{2}])
& \text{if $\qc$ does not divide $n$}, \\
\1_n & \text{if $\qc$ divides $n$}.
\end{array}
\right.
\end{equation}

If we look at Proposition \ref{ssV1}, we also have the following property
(for arbitrary $e\>1$). 

\begin{rema}
For $n\>2$, if $\qc$ does not divide $n$,
then $\LL_n$ is an irreducible quotient of $\V_n$.
\end{rema}

\section{Computing the derivatives of $\LL_n$ and $\MM_n$}
\label{CompDer}

In this section, we assume that $e$ is arbitrary. 
Remind 
(see \eqref{DEFVN}, \eqref{DEFPIN} and \eqref{DEFLAN})
that we have defined representations $\V_n$, $\MM_n$ and $\LL_n$ 
for all $n\>2$.
By Propositions \ref{P61} and \ref{ssV1}, we have: 
\begin{equation}
\label{DECV}
[\VV_n] = 
\left\{
\begin{array}{ll}
\Pi_n + \nu_n
& \text{if $\qc$ does not divide $n$}, \\
\Pi_n + \nu_n + \1_n & \text{if $\qc$ divides $n$},
\end{array}
\right.
\end{equation}
in the Grothendieck group of finite length representations of $\G_n$. 
Let us compute the derivatives of $\MM_n$.

\begin{lemm}
\label{derPi}
Suppose that $n\>2$.
\begin{enumerate}
\item
If $f=n=2$, the derivative $\MM_2^{(1)}$ is zero. 
\item 
Otherwise we have:
\begin{equation*}
\MM_n^{(1)} =
\left\{
\begin{array}{ll}
\1_{n-2}\times\nu^{(n+1)/2}
& \text{if $\qc$ does not divide $n$}, \\
\LL_{n-1}^\vee\cdot\nu^{1/2}
& \text{if $\qc$ divides $n$}.
\end{array}
\right. 
\end{equation*}
\item 
We have $\MM_n^{(2)}=\1_{n-2}$
and $\MM_n^{(k)}$ is zero for all $k\>3$.
\end{enumerate}
\end{lemm}

\begin{proof}
By Leibniz's rule (see Lemma \ref{dercal}(3)), we have: 
\begin{equation*}
[\V_n^{(1)}] = [\1_{n-2}\times\nu^{(n+1)/2}] + \nu_{n-1}^{1/2}
\end{equation*}
in the Grothendieck group of finite length representations of {$\G_{n-1}$}. 
Since the $k^{{\rm th}}$ derivative of a char\-acter is zero for $k\>2$, we 
have $\V_n^{(2)}=\1_{n-2}$ and $\V_n^{(k)}$ is zero for all $k\>3$.
The $k^{{\rm th}}$ deri\-va\-ti\-ve functors being exact,
the expected result follows from \eqref{DECV} together with Propositions 
\ref{P61} and \ref{ssV1}.
\end{proof}

\begin{coro}
\label{derLa}
Suppose that $n\>2$.
\begin{enumerate}
\item
We have:
\begin{equation*}
\LL_n^{(1)} =
\left\{
\begin{array}{ll}
\1_{n-2}\times\nu^{(n+1)/2}
& \text{if $\qc$ does not divide $n$}, \\
\nu_{n-1}^{-1/2} & \text{if $\qc$ divides $n$}.
\end{array}
\right. 
\end{equation*}
\item 
The second derivative $\LL_n^{(2)}$ is equal to $\1_{n-2}$ if $\qc$ does not 
divide $n$, and is zero otherwise. 
\item
The $k^{{\rm th}}$ derivative $\LL_n^{(k)}$ is zero for all $k\>3$.
\end{enumerate}
\end{coro}

\begin{rema}
\label{derPicor}
Since $\MM_n^{\vee}=\MM_n^{}\cdot\nu^{-1}$ by Properties {\bf P4} 
and {\bf P5}, we get the derivatives of 
$\MM_n^\vee$ and $\LL_n^\vee$ from Lemma \ref{derPi} and Corollary 
\ref{derLa}.
\end{rema}

\begin{exem}
\label{derst}
\begin{enumerate}
\item 
We have $\St_2=\Z([-1/2]+[1/2])=\MM_2\cdot\nu^{-1}$.
If $\qc=2$, the representation $\St_2$ is cuspidal thus its first derivative is zero.  
Otherwise, we have $(\St_2)^{(1)}=\nu^{1/2}$.
\item 
Let $\St_3$ denote the nondegenerate irreducible subquotient of 
$\nu^{-1}\times\1\times \nu$, that is:
\begin{equation*}
\St_3 = \Z([-1]+[0]+[1])
\end{equation*}
(see \cite[\S8]{MSc}).
If $\qc=3$, then $\St_3$ is cuspidal (\cite[\S6]{MSc})
thus its first and second derivatives are 
zero.  
If $\qc\neq3$, then:
\begin{equation*}
[\nu^{-1} \times \1 \times \nu] = \1_3 + \LL_3\cdot\nu^{-1} + 
(\LL_3)^{\vee}\cdot\nu + \St_3 
\end{equation*}
in the Grothendieck group of finite length representations of $\G_3$. 
We thus get $(\St_3)^{(1)} = \St_2\cdot\nu^{1/2}$ and 
$(\St_3)^{(2)} = \nu$. 
\end{enumerate}
\end{exem}

\section{A modular version of Badulescu-Lapid-M\'\i nguez's 
juxtaposition criterion} \label{blmanalog}

In Paragraph \ref{repmodl} we have defined $\Z(\Delta)$ for $\Delta$ a 
segment.
In \cite{MSc} an irreducible representation $\L(\Delta)$ is also introduced.
We will need it only for segments of length $\<2$.

\begin{defi}
Let $a$ be a half-integer.
Then $\L([a])=\Z([a])=\nu^{a}$ and:
\begin{equation*}
\L([a,a+1]) = 
\left\{
\begin{array}{ll}
\UIQ(\nu^a\times\nu^{a+1})&\text{if $e>1$},\\
\LL_{2}\cdot\nu^{a+1/2} & \text{if $e=1$}.
\end{array}
\right.
\end{equation*}
\end{defi}

\begin{rema}
Note that we have $\rp_{(1,1)}(\L([a,a+1]))=\nu^{a+1}\otimes\nu^{a}$ for all 
$a\in\frac{1}{2}\ZZ$.
\end{rema}

If we write $\St_2$ for the Steinberg representation of $\G_2$ as in Paragraph
\ref{defqc}, then we have: 
\begin{equation*}
\L([a,a+1]) =
\left\{
\begin{array}{ll}
\St_2\cdot\nu^{a+1/2}&\text{if $\qc\neq2$},\\
\nu^{a-1/2}&\text{if $\qc=2$}.
\end{array}
\right.
\end{equation*}

Note that $\LL_2=\St_2\cdot\nu$ if $\qc\neq2$.

\begin{lemm}[\cite{MSc}, Théorème 7.26]
\label{LinkedL}
Given two segments $\Delta$ and $\Delta'$ of length $\<2$, 
the representation $\L(\Delta)\times\L(\Delta')$ is irreducible 
if and only if $\Delta$ and $\Delta'$ are not linked.
\end{lemm}

Following \cite[Définition 2.1]{BLM}, say that two segments $[a,b]$ and $[c,d]$ are 
\textit{juxtaposed} if we have $c\equiv b+1$ or $a\equiv d+1$.

\begin{prop}\label{blm}
Let $\Delta,\Delta'$ be two segments.
Assume $\Delta'$ has length $2$.
Then $\Z(\Delta)\times\L(\Delta')$ is reducible 
if and only if $\Delta$ and $\Delta'$ are juxtaposed. 
\end{prop}

\begin{proof}
By twisting by a character, 
we may and will assume that $\Delta'=[0,1]$.
If $e=2$, then $\L([0,1])$ is cuspidal (see Lemma \ref{einstein}), 
thus the result is true by Proposition \ref{Lanzmann}. 
If $\Delta$ and $[0,1]$ are juxtaposed, Zelevinski's argument 
(see \cite[\S2]{Z})
proves that $\pi$ is reducible. 
In particular, if $e=1$, then $\Delta$ and $[0,1]$ are always juxtaposed. 
We are thus reduced to prove the following. 

\begin{prop}
\label{BLM01}
Assume that $e>2$ and let $a\<b$ be integers. 
Then $\Z([a,b])\times\L([0,1])$ is reducible 
if and only if $e$ divides $b+1$ or $a-2$.
\end{prop}

We write $\pi=\Z([a,b])\times\L([0,1])$. 
When $e$ divides $b+1$ or $a-2$, Zelevinski's argument proves that 
$\pi$ is reducible. 

We now assume $e$ does not divide $b+1$ nor $a-2$, and write $\equiv$ 
for the relation of congruence mod $e$ in $\ZZ$.
We thus have $a\nequiv2$ and $b\nequiv-1$.
The proof is by induction on $n=b-a+1$.

If $n=1$ then $\pi=\nu^a\times\L([0,1])$ and the result follows from 
Lemma \ref{LinkedL} since the segments $[a]$ and $[0,1]$ are not linked. 

Assume now $n\>2$. 
Our goal is to find irreducible representations $\s,\tau$, of degree $u,v$ 
res\-pectively, such that $\pi$ 
occurs as a subrepresentation of $\s\times\tau$ and as a quotient of 
$\tau\times\s$, and such that $\s\otimes\tau$ occurs with multiplicity 
$1$ in $\rp_{(u,v)}(\s\times\tau)$.
We will distinguish the following cases: 
\begin{enumerate}
\item\label{Es1}
$a\nequiv-1,1$ 
\item\label{Es2}
$a\equiv-1$
\item\label{Es3} 
$a\equiv1$ and $b\nequiv0,2$
\item\label{Es4} 
$a\equiv1$ and $b\equiv0,2$ and $e>3$
\item\label{Es5} 
$a\equiv1$ and $b\equiv0$ and $e=3$
\end{enumerate}

In Case \ref{Es1}, since $a\nequiv1$ and thanks to the inductive hypothesis, 
$\pi$ embeds in:
\begin{equation}
\label{I1}
\nu^a\times\Z([a+1,b])\times\L([0,1])\simeq 
\nu^a\times\L([0,1])\times\Z([a+1,b])
\end{equation}
and $\nu^a\times\L([0,1])$ is irreducible because $a\nequiv-1$. 
Since $\Z([a,b])$ is a quotient of $\Z([a+1,b])\times\nu^a$, 
we can choose $\s=\nu^a\times\L([0,1])$ and $\tau=\Z([a+1,b])$.
We compute the multiplicity 
of $\s\otimes\tau$ in $\rp_{(3,n-1)}(\s\times\tau)$
by applying the geometric lemma.
For this multiplicity to be $1$, it is enough to prove that $\s$ does not 
occur as a subquotient of the following representations: 
\begin{enumerate}
\item[(\ref{Es1}.1)]
$\nu^a\times1\times\nu^{a+1}$;
\item[(\ref{Es1}.2)] 
$\L([0,1])\times\nu^{a+1}$;
\item[(\ref{Es1}.3)] 
$\nu^a\times\Z([a+1,a+2])$;
\item[(\ref{Es1}.4)]
$\nu\times\Z([a+1,a+2])$;
\item[(\ref{Es1}.5)]
$\Z([a+1,a+3])$.
\end{enumerate}
This follows from \cite[Théorème 8.16]{MSc}.

In Case \ref{Es2}, Equation \eqref{I1} in addition with 
the fact that $\L([0,1])$ embeds in $\nu\times1$ 
implies that $\pi$ is a sub\-representation of:
\begin{equation*}
\nu^{-1}\times\nu\times1\times\Z([a+1,b]).
\end{equation*}
But $\pi$ is also a quotient of:
\begin{equation*}
\Z([a,b])\times1\times\nu\simeq1\times\Z([a,b])\times\nu
\end{equation*}
which itself is a quotient of the representation 
$1\times\Z([a+1,b])\times\nu^{-1}\times\nu$.
We thus can choose $\s=\nu^{-1}\times\nu$ and $\tau=1\times\Z([a+1,b])$.
Again, by the geometric lemma, 
it is enough to prove that $\s$ does not occur as a sub\-quo\-tient of
$\nu^{-1}\times1$, $\nu\times1$, $\Z([0,1])$ or $1\times1$,
which follows easily. 

In Case \ref{Es3}, we embed $\Z([a,b])$ into $\Z([a,b-1])\times\nu^b$ 
and show by a similar argument that we can choose 
$\s=\Z([1,b-1])\times\L([0,1])$ and $\tau=\nu^b$.
By using the geometric lemma, 
it is enough to prove that $\nu^b$ is different from $1$ and $\nu^{b-1}$,
which is immediate. 

In  Case \ref{Es4}, we prove the following more general lemma.

\begin{lemm}
Assume $e>3$.
Then $\Z([1,b])\times\L([0,1])$ is irreducible 
for any $b\>1$, $b\nequiv-1$.
\end{lemm}

\begin{proof}
We first treat the case where $b=2$ 
(the case where $b=1$ has already been done).
We embed $\pi=\Z([1,2])\times\L([0,1])$ in:
\begin{equation*}
\Z([1,2])\times\nu\times1\simeq
\nu\times\Z([1,2])\times1\hookrightarrow
\nu\times\nu\times \nu^2\times1
\end{equation*}
and we choose $\s=\nu\times\nu$ and $\tau=\nu^2\times1$. 

Now assume $b\>3$.
We embed $\Z([1,b])$ in 
$\Z([1,2])\times\Z([3,b])$ and then choose
$\s=\Z([1,2])$ and $\tau=\Z([3,b])\times\L([0,1])$.
By the geometric lemma, it is enough to prove $\s$ does not occur 
in: 
\begin{enumerate}
\item[(\ref{Es4}.1)]
$\nu\times\nu^3$;
\item[(\ref{Es4}.2)]
$\nu\times\nu$;
\item[(\ref{Es4}.3)]
$\L([0,1])$;
\item[(\ref{Es4}.4)]
$\nu^3\times\nu$;
\item[(\ref{Es4}.5)]
$\Z([3,4])$.
\end{enumerate}
This is immediate. 
\end{proof}

In Case \ref{Es5},
$n$ is of the form $3k$ for some $k\>1$, and we write 
$\Om_k=\Z([1,3k])$.

\begin{lemm}
The representation $\Om_1\times\L([0,1])$ is irreducible. 
\end{lemm}

\begin{proof}
Let $\xi$ be an irreducible subquotient of $\pi=\Om_1\times\L([0,1])$. 
It is thus a subquotient of the representation $\Z([1,3])\times\nu\times\1$.
By using Properties {\bf P2} and {\bf P3}, 
we deduce that $\xi$ is of the form $\Z(\m)$ where 
$\m$ is a multisegment in the following list: 
\begin{enumerate}
\item[(\ref{Es5}.1)]
$\m=[0,4]$;
\item[(\ref{Es5}.2)] 
$\m=[0,3]+[1]$;
\item[(\ref{Es5}.3)] 
$\m=[1,4]+[0]$;
\item[(\ref{Es5}.4)] 
$\m=[0,2]+[3,4]$;
\item[(\ref{Es5}.5)] 
$\m=[2,4]+[0,1]$;
\item[(\ref{Es5}.6)] 
$\m=[1,3]+[0,1]$;
\item[(\ref{Es5}.7)]
$\m=[1,3]+[0]+[1]$.
\end{enumerate}
We will prove that Case \ref{Es5}.7 is the only possible case, which implies that 
$\Om_1\times\L([0,1])$ is irredu\-ci\-ble and equal to $\Z([1,3]+[0]+[1])$.
By the geometric lemma, we get:
\begin{equation*}
\sy{\rp_{(3,2)}(\pi)} = 
\Z([1,3])\otimes\L([0,1])+(\Z([1,2])\times\nu)\otimes(\1\times\1)+
(\nu\times\L([0,1]))\otimes\Z([2,3])
\end{equation*}
and each of these three subquotients is irreducible. 
Since $\rp_{(3,2)}(\Z([0,4]))=\Z([0,2])\otimes\Z([3,4])$, 
we see that 
$\Z([0,4])$ cannot occur as a subquotient of $\pi$.

Now the semi-simplification of $\rp_{(1,2,2)}(\pi)$ is equal to:
\begin{equation*}
\begin{split}   
\nu\otimes\Z([2,3])\otimes\L([0,1])&+
\nu\otimes\L([0,1])\otimes\Z([2,3])+
\nu\otimes(\nu\times\1)\otimes\Z([2,3])\\
&+
\nu\otimes\Z([1,2])\otimes(\1\times\1)+
\nu\otimes(\nu^2\times\nu)\otimes(\1\times\1).
\end{split}
\end{equation*}
By using Proposition \ref{JacquetMu}, we see that Cases \ref{Es5}.4, 
\ref{Es5}.5 and \ref{Es5}.6 cannot occur.

Now the semi-simplification of $\rp_{(1,1,1,2)}(\pi)$ is equal to:
\begin{equation*}
\begin{split}   
\nu\otimes\1\otimes\nu\otimes\Z([2,3])&+
\nu\otimes\nu^2\otimes\1\otimes\L([0,1])+
\nu\otimes\nu^2\otimes\nu\otimes(\1\otimes\1)\\
&+
2\cdot\big(\nu\otimes\nu\otimes\1\otimes\Z([2,3])\big)+
2\cdot\big(\nu\otimes\nu\otimes\nu^2\otimes(\1\times\1)\big).
\end{split}
\end{equation*}
By using Proposition \ref{JacquetMu}, we see that Case \ref{Es5}.2  cannot occur. 

It remains to treat Case \ref{Es5}.3. 
The semi-simplification of $\rp_{(1,1,3)}(\Z([1,4])\times 1)$ is equal to: 
\begin{equation*}
\begin{split}   
\nu\otimes\nu^2\otimes (\Z([0,1])\times 1)+
\1\otimes\nu\otimes \Z([2,4])+
\nu\otimes\1\otimes \Z([2,4]).
\end{split}
\end{equation*}
By Proposition \ref{P61}(2) and the geometric lemma, we get:
\begin{equation*} [\rp_{(1,1,3)}(\Z([1,4])+ [0])]= \nu\otimes
\nu^2\otimes (\Z([0,1])\times 1)+\nu\otimes\1\otimes \Z([2,4]).
\end{equation*}
On the other hand, the semisimplification of $\rp_{(1,1,3)}(\pi)$ is equal to:
\begin{equation*}
\nu\otimes \1 \otimes \Z([1,3]) + 2\cdot (\nu\otimes \nu \otimes 
(\Z([2,3])\times 1)) + \nu\otimes \nu^2 \otimes (1\times \L([0,1]))
\end{equation*}
and each of the individual subquotients is irreducible. 
Therefore, Case \ref{Es5}.3 cannot occur. 
\end{proof}

The proof is now by induction on $k$.
We embed $\Om_{k+1}$ into $\Om_{1}\times\Om_{k}$
and choose $\s=\Om_1$ and $\tau=\L([0,1])\times\Om_{k}$.
By using the geometric lemma, we have to prove that, for all 
$0\<i\<2$, the factor $\s\otimes\tau$ does not occur as a subquotient 
of any of these three representations: 
\begin{enumerate}
\item[(\ref{Es5}.A)]
$\Z([1,i])\times\L([0,1])\times\Z([1,1-i])\otimes\Z([i+1,3])\times\Z([2-i,3k])$;
\item[(\ref{Es5}.B)]
$\Z([1,i])\times\nu\times\Z([1,2-i])\otimes\Z([i+1,3])\times1\times\Z([3-i,3k])$;
\item[(\ref{Es5}.C)]
$\Z([1,i])\times\Z([1,3-i])\otimes\Z([i+1,3])\times\L([0,1])\times\Z([4-i,3k])$.
\end{enumerate}
This follows by using Property {\bf P3}.
(Notice that the term (\ref{Es5}.A) does not appear if $i=2$).

This ends the proof of Proposition \ref{BLM01}.
\end{proof}

\section{Distinguished representations}\label{distprelims}

For $n\>2$, we write $\H_{n}$ for the subgroup of $\G_n$ made of all 
matrices of the form:
\begin{equation*}
\begin{pmatrix}
g&0\\
0&1
\end{pmatrix},
\quad
g\in\G_{n-1}.
\end{equation*}

\begin{defi}
Assume that $n\>2$.
A smooth $\R$-representation $(\pi,\V)$ of $\G_n$ is said to be 
$\H_n$-distinguished if $\V$ possesses a nonzero $\H_n$-invariant 
linear form. 
\end{defi}

If the space $\Hom_{\H_n}(\V,\R)$ has finite dimension over $\R$, we denote 
this dimension by $d(\pi)$. 

\subsection{Cuspidal representations} 
\label{sectioncusp}

Just as in the complex case (see \cite{P}), we have the following result. 

\begin{theo}
\label{THEOCUSP}
Let $n\>2$ and let $\rho\in\GH_{n}$ be a cuspidal representation.
Then $\rho$ is distinguished if and only if $n=2$.
When it is the case, we have $d(\rho)=1$.
\end{theo}

\begin{proof}
Write $\P_{n}$ for the mirabolic subgroup of $\G_{n}$, that is the subgroup 
made of all matrices with last row $(0,\dots,0,1)$.
By \cite[{III},Theorem 1.1]{Vigb}, the restriction of $\rho$ to $\P_{n}$ is 
isomorphic, just as in the complex case, to
the compact $\R$-induction:
\begin{equation*}
\ind^{\P_{n}}_{\U_{n}}(\psi_{n})
\end{equation*}
of a generic character $\psi_{n}$ of the standard maximal unipotent subgroup
$\U_{n}$ of $\G_{n}$. 
As $\P_{n}=\H_{n}\U_{n}$, the restriction of $\rho$ to $\H_{n}$ is 
iso\-morphic to the compact $\R$-induction 
$\ind^{\H_{n}}_{\H_{n}\cap\U_{n}}(\psi_{n})$, which carries a nonzero 
$\H_{n}$-fixed $\R$-linear form if and only if $\psi_{n}$ is trivial on 
$\H_{n}\cap\U_{n}$.
This happens if and only if $n=2$, in which case we 
have $d(\rho)=\dim\Hom_{\H_2\cap\U_2}(\psi_{2},1)=1$.
\end{proof}

\subsection{Distinction and contragredient}

We have the very useful following result.
Assume $n\>2$.

\begin{prop}
\label{GK}
Assume that the characteristic of $\R$ is not $2$, 
and let $\pi\in\GH_n$.
Then $\pi$ is $\H_n$-distinguished if and only $\pi^\vee$ is. 
\end{prop}

\begin{proof}
In the complex case, this result follows from the well-known fact
(due to Gelfand and Kazhdan)
that the contragradient of an irreducible representation $\pi\in\GH_n$ is 
isomorphic to $\pi\circ\s$ where $\s$ is the involution:
\begin{equation*}
g\mapsto\text{transpose of }g^{-1}
\end{equation*}
 of $\G_n$.
For a field $\R$ of characteristic not $2$, 
this argument still holds 
(see \cite[Remarque 2.7]{MSc}).
\end{proof}

\begin{rema}
Note that the condition $\ee>1$ implies that the characteristic of $\R$ is 
not $2$. 
\end{rema}

In the case where $\R$ has characteristic $2$, 
it is likely that the Gelfand-Kazhdan property still holds
(thus so Proposition \ref{GK} does) but we won't need it 
in this case. 

\begin{prop}
\label{GKext}
Assume that the characteristic of $\R$ is not $2$ and $n=n_1+n_2$ where $n_1, n_2$ are positive integers. 
Let $\pi_i\in\GH_{n_i}$ for $i=1,2$. Then $\pi_1 \times \pi_2$ is $\H_n$-distinguished if and only if 
$\pi_2^{\vee}\times \pi_1^{\vee}$ is.  
\end{prop}

\begin{proof}
In the complex case, the proof follows from the above-mentioned theorem of Gelfand-Kazhdan,
the existence of a group 
automorphism of $\G_n$ which maps $\P_{(n_1,n_2)}$ to $\P_{(n_2,n_1)}$, 
$\pi_1\times \pi_2$ to $\pi_2^{\vee}\times \pi_1^{\vee}$ and $\G_{n-1}$ 
to a conjugate of $\G_{n-1}$. 
All these are valid for $\R$ 
(see \cite[III.1.16]{Vigb} and \cite[1.9 ]{Z}) 
if it has characteristic not $2$. 
Hence:
\begin{equation*}
\Hom_{\G_{n-1}}(\pi_1\times \pi_2, \R) \simeq 
\Hom_{\G_{n-1}}(\pi_2^{\vee}\times \pi_1^{\vee},\R)
\end{equation*}
and our claim follows.
\end{proof}

\subsection{The Bernstein-Zelevinski filtration}
\label{BZF}

For $i\in\{0,1,\dots,n\}$, we write $\R_{i,n}$ for the subgroup of matrices 
of $\G_n$ of the form: 
\begin{equation*} 
\begin{pmatrix} g & * \\ 0 & h \end{pmatrix}
\end{equation*}
such that $g\in\G_{i}$ and $h$ is an upper triangular and unipotent matrix
of $\G_{n-i}$. 
In particular, $\R_{0,n}$ is the standard maximal unipotent subgroup 
$\U_{n}$ of $\G_{n}$ and $\R_{n-1,n}$ is the mirabolic subgroup $\P_n$ of 
$\G_n$.
Fix a nontrivial smooth character $\psi:\F\to\mult\R$ and, 
for $i\in\{0,1,\dots,n-1\}$, 
write $\psi_i$ for the generic character of $\U_{i}$ defined by:
\begin{equation*}
\psi_i(h) = \psi(h_{1,2}+\dots+h_{i-1,i})
\end{equation*}
for all $h\in\U_i$.
From \cite[III.1.3]{Vigb}, we have the following result. 

\begin{theo}
\label{BZfilt}
Let $\V$ be a representation of $\G_n$.
There are $\P_n$-stable subspaces $\V_0,\dots,\V_n$ of $\V$ such that 
$\{0\}=\V_0\subseteq\V_1\subseteq\dots\subseteq\V_n=\V$ and:
\[\V_{i+1}/\V_{i}\simeq \ind_{\R_{i,n}}^{\P_n}(\V^{(n-i)}\nu_i^{1/2}\otimes\psi_{n-i})\]
for all $i\in \{0,...,n-1\}$.
\end{theo}

As in the complex case (see page 54 of 
\cite{Flicker} and \cite[Proposition 1]{P}), 
we get the following result by using the Bernstein-Zelevin\-ski filtration.

\begin{lemm}
\label{BZapp}
Let $\pi$ be a smooth representation of $\G_n$ with $n\>3$, 
and assume that:
\begin{enumerate}
\item
$\pi^{(1)}$ does not have any quotient isomorphic to $\nu_{n-1}^{-1/2}$;
\item
$\pi^{(2)}$ does not have any quotient isomorphic to $\1_{n-2}$.
\end{enumerate}
Then $\pi$ is not distinguished. 
\end{lemm}

\subsection{The Three Orbits Lemma}

As in the complex case \cite{V}, we have the following very useful lemma.  

\begin{lemm}
\label{TOL}
Let $n\>2$ and $k\in\{1,\dots,n-1\}$ be integers, 
and let $\rho\in\GH_k$ and $\tau\in\GH_{n-k}$.
Assumme $\rho\times\tau$ is $\H_{n}$-distinguished.
Then at least one of the following conditions is satisfied: 
\begin{enumerate}
\item[(A)] 
$\rho=\nu_{k}^{(n-2-k)/2}$ and $\tau\cdot\nu^{k/2}$ is $\H_{n-k}$-distinguished. 
\item[(B)]
$\rho\cdot\nu^{-(n-k)/2}$ is $\H_{k}$-distinguished and $\tau=\nu_{n-k}^{-(k-2)/2}$. 
\item[(C)]
$\rho^{(1)}\cdot\nu^{-(n-1-k)/2}$ and $\tau^{*(1)}\cdot\nu^{-(k-1)/2}$
have a trivial quotient. 
\end{enumerate}
Conversely, if $\rho\in\GH_k$ and $\tau\in\GH_{n-k}$ satisfy 
$(\A)$ or $(\B)$, then $\rho\times\tau$ is $\H_{n}$-distinguished.  
\end{lemm}

\begin{proof}
The proof in just as in the complex case (see \cite[Section 5]{V}).
\end{proof}

\begin{rema}\label{TOLext}
Notice that if 
$\tau$ is smooth (not necessarily irreducible),
we still have conditions similar to Lemma \ref{TOL}. 
We will have the occasion to use this in the 
case where:
\begin{equation*}
\rho=\nu_{n-3}^{1/2},
\quad
\tau=\St_2\cdot \nu^{-(n-2)/2}\times \chi,
\quad\chi\in \GH_1.
\end{equation*}
In this case, 
$\rho\times \tau$ is $\H_n$-distinguished
if and only if $\tau\cdot \nu^{(n-3)/2}$ is $\H_3$-distinguished. 
\end{rema}

\begin{coro}
\label{reduccusp}
Let $n\>3$, and let $\pi\in\GH_n$ be $\H_{n}$-distinguished.  
Then one of the following properties holds: 
\begin{enumerate}
\item
$\pi=\1_{n-2}\times\tau$ for some irreducible cuspidal 
representation $\tau\in\GH_2$.
\item
The cuspidal support of $\pi$ is made of characters of $\G_1$. 
\end{enumerate}
\end{coro}

\begin{proof}
There are irreducible cuspidal representations $\tau_1,\dots,\tau_r$ such that 
$\pi$ is a quotient of $\tau_1\tdt\tau_r$.  
Since $n\>3$, Theorem \ref{THEOCUSP} implies that $\pi$ is not cuspidal, 
which implies that $r\>2$.
Let $k$ denote the largest integer among the $\deg(\tau_i)$'s and let $\tau_i$
have degree $k$ with $i$ maximal for this property. 
Then by \cite{MSt} and Lemma \ref{PERM}, one may assume that $i=r$. 
Now write $\tau=\tau_r$ and let $\rho$ be an irreducible subquotient of 
$\tau_1\tdt\tau_{r-1}$ such 
that $\pi$ is a quotient of $\rho\times\tau$. 
Since $\pi$ is distinguished, so $\rho\times\tau$ is.
Apply Lemma \ref{TOL} to this product. 
According to Theorem \ref{THEOCUSP}, we obtain that $k$ must be $\<2$.
Moreover, if $k=2$, then $\rho=\1_{n-2}$. 
\end{proof}

\subsection{Distinction of the twists of $\LL_n$ and $\MM_n$}

We first determine which twists of $\LL_n$ are distinguished. 

\begin{lemm}
\label{Ldistchar}
Let $n\>2$ and $\chi\in\GH_1$.
Then $\LL_n\cdot\chi$ is distinguished if and only if $\chi=1$.
\end{lemm}

\begin{proof}
If $\qc$ divides $n$, then $\LL_n$ is the trivial character 
and the result is immediate. 
If $\qc$ does not divide $n$, then we have the exact sequence:
\begin{equation*}
\label{EXSEQLL}
0 \to \nu_n\cdot\chi \to \V_n\cdot\chi = 
(\nu_{n-1}^{1/2}\cdot\chi)\times\nu^{(n+1)/2}\chi 
\to \LL_n\cdot\chi \to 0
\end{equation*}
By Lemma \ref{TOL} with $k=n-1$, the representation $\V_n\cdot\chi$, 
and hence $\LL_n\cdot\chi$, is non-distinguished for 
$\chi\notin\{1,\nu^{-1}\}$. 
If $\chi=\nu^{-1}$ is non-trivial (which forces $e>1$), 
then Lemma \ref{BZapp} together with 
Corollary \ref{derLa} imply that $\LL_n\cdot\nu^{-1}$ is not distinguished. 
Now assume that $\chi=1$. 

If $e>1$, the contragredient $\V_n^{\vee}$ is distinguished by Lemma \ref{TOL} 
(A) but $\nu_n^{-1}$ is not. 
Thus $\V_n^{\vee}$ carries a nonzero invariant linear form vanishing on 
$\nu_n^{-1}$. 
It thus gives a nonzero invariant linear form on the subrepresentation
$\LL_n^{\vee}$. By Proposition \ref{GK}, the representation 
$\LL_n$ is distinguished.

If $e=1$, then $\V_n=\1_n\oplus\LL_n$ by Pro\-posi\-tion \ref{ssV1}. 
By Lemma \ref{TOL}, we have $d(\V_n)\>2$ since Conditions (A) and (B) 
are fulfilled.
Thus $\LL_n$ is distinguished with $d(\LL_n)=d(\V_n)-1$. 
\end{proof}

\begin{coro}
\label{Converse}
Assume that $e>1$.
All the irreducible representations of $\G_n$, $n\>3$ 
in the list given by Theorem
\ref{MAINTHEOREM} are distinguished.
\end{coro}

\begin{proof}
When applied with $k=n-1$ and $k=n-2$ respectively,
Lemma \ref{TOL} gives the result for 
$\nu_{n-1}^{-1/2}\times\chi$ and $\1_{n-2}\times\tau$.
By passing to the contragredient (Proposition \ref{GK}), 
we get the result for the representation $\nu_{n-1}^{1/2}\times\chi$
when $e>1$.

By Lemma \ref{Ldistchar}, $\LL_n$ is distinguished.
By passing to the contragredient,
we get the result for $\LL_n^\vee$ when $e>1$.
(Note that $\LL_n$ is selfdual when $e=1$.)
This finishes the proof. 
\end{proof}

We now determine which twists of $\MM_n$ are distinguished. 
This is done in Lemma \ref{Ldistchar} when $\qc$ does not divide $n$. 
We now treat the case where $\qc$ divides $n$.

\begin{lemm}
\label{Sndist}
Assume that $e$ is not $1$ and divides $n$. 
For $\chi\in \GH_1$, 
the representations $\MM_n\cdot\chi$ and $\MM_n^{\vee}\cdot\chi$ are not 
distinguished.
\end{lemm}

\begin{proof}
By Proposition \ref{GK}, it is enough to prove it for $\MM_n^{\vee}\cdot\chi$.
By Lemma \ref{BZapp}, for $\MM_n^{\vee}\cdot\chi$ to be distinguished, 
it is necessary that at least one of the derivatives 
$(\MM_n^{\vee}\cdot\chi)^{(i)}$ for $i=1,2$ has a character as a quotient.  
We have:
\begin{equation*}
(\MM_n^{\vee}\cdot\chi)^{(1)}=\LL_{n-1}^{\vee}\cdot \chi\nu^{-1/2},
\quad
(\MM_n^{\vee}\cdot\chi)^{(2)}= \nu_{n-2}^{-1}\cdot\chi.
\end{equation*}
By Lemma \ref{BZapp}, we conclude that $\MM_n^{\vee}\cdot\chi$ is not 
distinguished when $\chi\neq \nu$. 
It remains to consider $\MM_n^{\vee}\cdot\nu$, 
or rather its contragredient $\MM_n\cdot\nu^{-1}$.  
Its second derivative is $\nu_{n-2}^{-1}$. 
By Lemma \ref{BZapp}, our claim follows. 
\end{proof}

\begin{lemm}
\label{Pi_ndist}
Assume that $e=1$ and $\ell$ divides $n$.
For $\chi\in \GH_1$, 
the representations $\MM_n\cdot\chi$ is distinguished if and only if $\chi=1$.
\end{lemm}

\begin{proof}
When $e=1$, 
the representation $\MM_n$ is selfdual thus the first part of the proof 
of Lemma \ref{Sndist} still holds.
Thus $\MM_n\cdot\chi$ is not distinguished for any $\chi\neq1$. 
However, the second derivative of $\MM_n$ is $\1_{n-2}$, 
thus Lemma \ref{BZapp} is not sufficient to determine whether or not $\MM_n$ 
is distinguished.

Let $\H_n$ act on $\X=\P_{(n-1,1)}\backslash\G_n$. 
There are two closed 
orbit $\A$ and $\B$ in $\X$, where $\A$ is redu\-ced to a point and $\B$ is 
isomorphic to $\P_{(n-2,1)}\backslash\G_{n-1}$ (see \cite[5]{V}).
Since $q$ is congruent to $1$ mod the characteristic of $\R$, 
the modulus $\R$-character of $\P_{(n-1,1)}$ is trivial.
By \cite[Proposition I.2.8]{Vigb}, 
there is a non-zero $\G_n$-invariant linear form $\upmu_\X$ on 
$\V_n$. 
Similarly, there is a non-zero $\H_n$-in\-va\-riant linear form on 
$\Cc^\ii_c(\B,\R)$.
Composition by the restriction from $\X$ to $\B$ gives us a 
non-zero $\H_n$-invariant linear form $\upmu_\B$ on $\V_n$.
Finally, for $f\in \Cc^\ii_c(\X,\R)$, we write $\upmu_{\A}(f)$ for the value 
of $f$ at $\A$. 
We thus get three $\H_n$-invariant linear forms on $\V_n$. 

The form $\upmu_\X$ is actual\-ly $\G_n$-equivariant; its image is $\1_n$, 
and its kernel $\W_n$ has length $2$, with socle $\1_n$ (the space of 
$\R$-valued constant functions on $\X$) and irreducible quotient $\MM_n$.
We claim that these three linear forms are linearly independent. 
Granting the claim, there is no nontrivial linear 
combination of $\upmu_\A,\upmu_\B$ that va\-nish\-es on $\W_n$.
Moreover, if $f_0$ denotes the constant function taking value $1$, 
and if $\upmu_\B$ is chosen so that $\upmu_\B(f_0)=1$, then:
\begin{equation*}
(\upmu_\A-\upmu_\B)(f_0)=0.
\end{equation*}
Therefore, $\upmu_\A-\upmu_\B$ is a nonzero $\H_n$-invariant linear 
forms on $\W_n$ that vanishes on the space of constant functions; 
it thus induces a nonzero $\H_n$-invariant linear 
forms on $\MM_n$.
Finally $\MM_n$ is $\H_n$-distinguished when $e=1$ and 
$\ell$ divides $n$.

It remains to prove the claim.
Let $\U$ denote the unique open $\H_n$-orbit in $\X$, 
so that $\X$ is the disjoint union of $\A$, $\B$ and $\U$, 
and let $\widehat{\U}$ be its preimage in $\G_n$.
Let $\upmu$ be a Haar measure on $\G_n$. 
Since $\G_n$ is locally pro-$p$,
there is a compact open subset $\Omega\subseteq\widehat{\U}$ 
with nonzero measure.
Write $\phi$ for the characteristic function of the image of $\Omega$ 
in $\X$.
By \cite[\S2.8]{Vigb}, there exists a $\a\in\mult\R$ such that:
\begin{equation*}
\upmu_\X(\phi) = \a\cdot\upmu(\1_\Omega) \neq0.
\end{equation*}
On the other hand, we have $\upmu_\A(\phi)=\upmu_\B(\phi)=0$ 
and hence the linear forms $\upmu_\X, \upmu_\A$ and $\upmu_\B$ are 
linearly independent.
\end{proof} 

\begin{rema}
\label{Pi_nmult2}
Suppose $e=1$ and $\ell$ does not divide $n$. 
It follows from the proof of Lemma \ref{Pi_ndist} that $d(\V_n)$ is at least 
3. 
On the other hand, the conditions of Lemma \ref{TOL} implies that there is at 
most one $\H_n$-invariant linear form upto scalars on each of three orbits 
$\A,\B$ and $\U.$ Thus, $d(\V_n)=3$. 
Since $\V_n=\1_n \oplus \Pi_n$, it follows that $d(\Pi_n)=2$.
\end{rema} 

\subsection{First reduction of the problem}
\label{firstaid}

Thanks to Corollary \ref{reduccusp}, we are already reduced to studying those 
$\H_n$-distinguished irredu\-cible representations of $\G_n$, with $n\>3$, 
whose cuspidal support is made of characters. 

\begin{lemm}
\label{Derivative}
Let $\rho\in\GH_k$ be such that $\rho^{(1)}\cdot\nu^{-(n-1-k)/2}$ has a trivial 
quotient.  
Then $\rho$ is one of the following representations:
\begin{enumerate}
\item
$\nu_{k-1}^{(n-k-1)/2}\times\mu$ with $\mu\in\GH_1-\{\nu^{(n-2k-1)/2},\nu^{(n-1)/2}\}$;
\item
$\nu_{k}^{(n-k)/2}$;
\item
$\LL_{k}^*\cdot\nu^{(n-k)/2}$.
\end{enumerate}
\end{lemm}

\begin{proof}
We follow the proof given in the complex case in \cite[Lemma 6.2]{V}. 
The condition on $\rho$ is equivalent to saying that $\rho$ embeds into a 
representation of the form:
\begin{equation*}
\V(\mu) = \nu^{(n-1-k)/2}_{k-1}\times\mu,
\quad
\mu\in\GH_1.
\end{equation*}
If $\mu\notin\{\nu^{(n-2k-1)/2},\nu^{(n-1)/2}\}$,
this representation is irreducible (see Proposition \ref{Lanzmann}) 
thus $\rho$ is as in Case 1. 
Assume that $e>1$.
Thanks to Lemma \ref{UIQ}, Proposition \ref{P61} and \eqref{CONTRAL}, 
we have: 
\begin{enumerate}
\item 
$\V(\nu^{(n-1)/2})$ has a 
unique irreducible subrepresentation, which is $\nu_{k}^{(n-k)/2}$. 
Thus $\rho$ is as in Case 2.
\item 
$\V(\nu^{(n-2k-1)/2})$ has a 
unique irreducible subrepresentation, which is $\LL_{k}^*\cdot\nu^{(n-k)/2}$.  
Thus $\rho$ is as in Cases 2 or 3.
\end{enumerate}
Assume now that $e=1$.
Then, by Proposition \ref{ssV1}, any subrepresentation $\rho$ of 
$\V(\nu^{(n-1)/2})$ is as in Case 2 or 3.
Note that, in the case where $\qc$ divides $k$, the representation $\V_k$ is 
indecomposable, thus $\rho$ must be the character $\nu_{k}^{(n-k)/2}$.
This finishes the proof of Lemma \ref{Derivative}.
\end{proof}

In conclusion, we have the following result. 

\begin{prop}
\label{reduction1}
Assume $n\>3$.
Let $\pi\in\GH_n$ be $\H_{n}$-distinguished. 
Then there are 
$\rho\in\GH_{n-1}$ and $\chi\in\GH_1$ such that 
$\pi$ is an irreducible quotient of $\rho\times\chi$
and at least one of the following conditions holds: 
\begin{enumerate}
\item 
One has $\rho=\nu_{n-1}^{-1/2}$ or $\rho=\nu_{n-1}^{1/2}$.
\item
One has $\rho={\LL}^\vee_{n-1}\cdot\nu^{1/2}$.
\item
One has $\rho=\1_{n-2}\times\mu$ for some 
$\mu\in\GH_1-\{\nu^{-(n-1)/2},\nu^{(n-1)/2}\}$. 
\item
The representation $\rho\cdot\nu^{-1/2}$ is $\H_{n-1}$-distinguished and 
$\chi=\nu^{-(n-3)/2}$. 
\end{enumerate}
Moreover, if $e>1$, then $\pi$ is the unique irreducible quotient of 
$\rho\times\chi$. 
\end{prop}

In order to prove our main theorem \ref{MAINTHEOREM}, 
our strategy is to study, by induction on $n\>2$, 
the irreducible quotients of $\rho\times\chi$
in all these cases when $e>1$,
and to prove that they are either in the list of Theorem 
\ref{MAINTHEOREM} or non-dis\-tin\-gui\-shed.

Assuming that Theorem \ref{MAINTHEOREM} holds for $\G_{n-1}$ with $n\>3$, 
we thus have to study the distinction of the following representations:

\vspace{.1cm}

Case 1:\
the irreducible quotients of $\nu_{n-1}^{-1/2}\times\chi$ and 
$\nu_{n-1}^{1/2}\times\chi$ for $\chi\in\GH_1$;

\vspace{.1cm}

Case 2:\
the irreducible quotients of $\1_{n-2}\times\mu\times\chi$
for $\mu\in\GH_1-\{\nu^{-(n-1)/2},\nu^{(n-1)/2}\}$, $\chi\in\GH_1$; 

\vspace{.1cm}

Case 3:\
the irreducible quotients of $\LL_{n-1}^\vee\cdot\nu^{1/2}\times\chi$
for $\chi\in\GH_1$;

\begin{enumerate}
\item[Case 4:]
the irreducible quotients of $\rho\times\nu^{-(n-3)/2}$ where $\rho$ is:

\begin{tabular}{lp{14cm}}
{(4.a)} &
the character $\nu_{n-1}^{1/2}$ (included in Case 1 above); \\
{(4.b)} & 
a representation $\1_{n-2}\times\mu$ with 
$\mu\in\GH_1-\{\nu^{-(n-1)/2},\nu^{(n-1)/2}\}$ 
(included in Case 2); \\
{(4.c)} &
a representation $\nu_{n-2}\times\mu$ with 
$\mu\in\GH_1-\{\nu^{-(n-3)/2},\nu^{(n+1)/2}\}$; \\
{(4.d)} &
a representation $\nu_{n-3}^{1/2}\times\tau$ with $\tau\in\GH_2$ 
infinite-dimensional; \\
{(4.e)} &
one of the representations $\LL_{n-1}\cdot\nu^{1/2}$ or 
$\LL_{n-1}^\vee\cdot\nu^{1/2}$
(see Case 3 above);
\end{tabular}
\end{enumerate}

Cases 1 and 4.c are treated in Section \ref{CQuot} for arbitrary $e\>1$, 
and Case 4.d is treated in Section \ref{GQuot} for $e>1$. 

We reduce Case 2 to studying 
$\UIQ(\mu\times\LL_{n-1}\cdot\nu^{-1/2})$ for 
$\mu\in\GH_1-\{\nu^{-(n-1)/2},\nu^{(n-1)/2}\}$
in Section \ref{TQuot}, when $e>1$.

In Section \ref{LQuot}, we do the remaining cases when $e>1$. 

\section{Computing the irreducible quotients of $\nu_{n-1}^{1/2}\times\chi$ 
  for $\chi\in\GH_1$} 
\label{CQuot}

\begin{lemm}
\label{L1}
Assume $e>1$.
Let $a,b\in\ZZ$ with $a\<b$.
For $\chi\in\GH_1$, write $\V(\chi)=\Z([a,b])\times\chi$.
\begin{enumerate}
\item
If $\chi\notin\{\nu^{a-1},\nu^{b+1}\}$, then $\V(\chi)$ is irreducible. 
\item
Assume that $\chi=\nu^{b+1}$ and $e$ does not divide $n$.
Then $\V(\nu^{b+1})$ has length $2$ and we have the following exact sequence: 
\begin{equation*}
0\to\Z([a,b+1])\to 
\V(\nu^{b+1})
\to\Z([a,b]+[b+1])\to0.
\end{equation*}
\item
Assume that $\chi=\nu^{a-1}$ and $e$ does not divide $n$.
Then $\V(\nu^{a-1})$ has length $2$ and we have the following exact sequence: 
\begin{equation*}
0\to\Z([a,b]+[a-1])\to\V(\nu^{a-1})
\to\Z([a-1,b])\to0.
\end{equation*}
\item
If $e$ divides $n$, then $\nu^{a-1}=\nu^{b+1}$ and 
$\V(\nu^{b+1})$ has length $3$ with:
\begin{equation*}
\UIS(\V(\nu^{b+1}))=\Z([a,b+1]),
\quad
\UIQ(\V(\nu^{b+1}))=\Z([a-1,b]).
\end{equation*}
\end{enumerate}
\end{lemm}

\begin{proof}
Case 1 follows from Propositions \ref{Lanzmann} and \ref{Linked}. 
The other cases reduce to Proposition \ref{P61} by twisting by the character 
$\chi^{-1}$, since 
$\V(\chi)\cdot\nu^{-c}=\Z([a-c,b-c])\times\chi\nu^{-c}$ 
for $c\in\ZZ$.
\end{proof}

From Lemma \ref{L1} we get the following proposition. 

\begin{prop}\label{L2}
Assume $e>1$.
For all $n\>1$, we have: 
\begin{equation*}
\UIQ(\nu_{n-1}^{1/2}\times\chi)=
\left\{
\begin{array}{ll}
\nu_{n-1}^{1/2}\times\chi & 
\text{if } \chi\notin\{\nu^{-(n-1)/2},\nu^{(n+1)/2}\}, \\
\1_{n} & \text{if } \chi=\nu^{-(n-1)/2}, \\
\LL_n & \text{if } \chi=\nu^{(n+1)/2}. 
\end{array}
\right.
\end{equation*}
\end{prop}

Twisting by $\nu^{-1}$, we get the following. 

\begin{prop}\label{L3}
Assume $e>1$.
For all $n\>1$, we have: 
\begin{equation*}
\UIQ(\nu_{n-1}^{-1/2}\times\chi)=
\left\{
\begin{array}{ll}
\nu_{n-1}^{-1/2}\times\chi & 
\text{if } \chi\notin\{\nu^{-(n+1)/2},\nu^{(n-1)/2}\}, \\
\nu_{n}^{-1} & \text{if } \chi=\nu^{-(n+1)/2}, \\
\LL_{n}\cdot\nu^{-1} & \text{if } \chi=\nu^{(n-1)/2}. 
\end{array}
\right.
\end{equation*}
\end{prop}

By duality, we get the following.

\begin{prop}\label{L4}
Assume $e>1$.
For all $n\>1$, we have: 
\begin{equation*}
\UIS(\nu_{n-1}^{-1/2}\times\chi)=
\left\{
\begin{array}{ll}
\nu_{n-1}^{-1/2}\times\chi & 
\text{if } \chi\notin\{\nu^{-(n+1)/2},\nu^{(n-1)/2}\}, \\
\1_{n} & \text{if } \chi=\nu^{(n-1)/2}, \\
\LL_{n}^{\vee} & \text{if } \chi=\nu^{-(n+1)/2}. 
\end{array}
\right.
\end{equation*}
\end{prop}

Twisting by $\nu$, we get the following.

\begin{prop}\label{L5}
Assume $e>1$.
For all $n\>1$, we have: 
\begin{equation*}
\UIS(\nu_{n-1}^{1/2}\times\chi)=
\left\{
\begin{array}{ll}
\nu_{n-1}^{1/2}\times\chi & 
\text{if } \chi\notin\{\nu^{-(n-1)/2},\nu^{(n+1)/2}\}, \\
\nu_{n} & \text{if } \chi=\nu^{(n+1)/2}, \\
\LL_{n}^{\vee}\cdot \nu & \text{if } \chi=\nu^{-(n-1)/2}. 
\end{array}
\right.
\end{equation*}
\end{prop}

In the case where $e=1$, we summarize below the results obtained in Section \ref{Castor}.

\begin{prop}\label{L6}
Assume $e=1$.
\begin{enumerate}
\item
If $\chi\neq\nu^{(n+1)/2}$, then $\nu_{n-1}^{1/2}\times\chi$ is irreducible. 
\item
If $\ell$ does not divide $n$, the irreducible quotients of 
$\nu_{n-1}^{1/2}\times\nu^{(n+1)/2}$ are $\1_n$ and $\MM_n$.
\item
If $\ell$ divides $n$, the irreducible quotient of 
$\nu_{n-1}^{1/2}\times\nu^{(n+1)/2}$ is $\1_n$.
\end{enumerate}
\end{prop}

Thus we have treated Case 1 of Proposition \ref{reduction1}.

\begin{coro}
Let $e>1$ and $\mu\in\GH_1-\{\nu^{-(n-3)/2}, \nu^{(n+1)/2}\}$. 
Then:
\begin{equation*}
\UIQ(\nu_{n-2}\times\mu\times\nu^{-(n-3)/2})=
\left\{
\begin{array}{ll}
\mu \times \nu_{n-1}^{1/2} & \text{if $\mu\neq \nu^{-(n-1)/2}$},\\
\LL_n^{\vee}\cdot\nu & \text{if $\mu= \nu^{-(n-1)/2}$}. 
\end{array}
\right. \\
\end{equation*}
\end{coro}

\begin{proof}
By assumption on $\mu$, the representation $\nu_{n-2}\times\mu$ is 
irreducible. 
It is thus isomorphic to $\mu\times\nu_{n-2}$. 
It thus suffices to consider the representation 
$\pi(\mu)=\UIQ(\mu\times \nu_{n-2}\times\nu^{-(n-3)/2})$.
By Proposition \ref{L2}, we have:
\begin{equation*}
\UIQ(\nu_{n-2}\times\nu^{-(n-3)/2}) = \nu_{n-1}^{1/2}
\end{equation*}
thus $\pi(\mu)$ is equal to $\UIQ(\mu\times\nu_{n-1}^{1/2})$.
By assumption on $\mu$, the representation 
$\mu\times\nu_{n-1}^{1/2}$ is reducible if and only if $\mu=\nu^{-(n-1)/2}$.
Finally, the representation:
\begin{equation*}
\pi(\nu^{-(n-1)/2}) 
= \UIQ(\nu^{-(n-1)/2}\times\nu_{n-1}^{1/2})
= \UIS(\nu_{n-1}^{1/2}\times\nu^{-(n-1)/2})
\end{equation*}
is equal to $\LL_n^{\vee}\cdot\nu$ by Pro\-po\-si\-tion \ref{L3}.
By Lemma \ref{Ldistchar}, it is not distinguished. 
\end{proof}

Thus we have treated Case 4.c of Proposition \ref{reduction1}.

\section{Computing $\UIQ(\1_{n-2}\times\mu\times\chi)$ for 
$\mu\in\GH_1-\{\nu^{-(n-1)/2},\nu^{(n-1)/2}\}$ and $\chi\in\GH_1$} 
\label{TQuot}

In this section, we fix a character $\mu\in\GH_1$ different from 
$\nu^{-(n-1)/2}$ and $\nu^{(n-1)/2}$, and we assume that $e>1$.
Note that this implies that 
$\1_{n-2}\times\mu=\mu\times\1_{n-2}$ is irreducible. 
For $\chi\in\GH_1$, write:
\begin{equation*}
\W(\chi)=\1_{n-2}\times\mu\times\chi.
\end{equation*}
We record below two facts in the form of the following 
lemma which will be used repeatedly in what follows.

\begin{lemm}\label{zero}
$\W(\chi)$ has a unique irreducible subrepresentation and a unique 
irreducible quotient. 
Moreover, one has:
\begin{equation*}
\UIQ(\chi\times\chi\nu)=
\left\{
\begin{array}{ll}
\St_2\cdot\chi\nu^{1/2} & \text{if $e>2$},\\
\1_2\cdot\chi\nu^{-1/2} & \text{if $e=2$}.
\end{array}
\right. 
\end{equation*}
\end{lemm}

\begin{proof}
The first statement follows from Lemma \ref{UIQ} and the second one from Lemma 
\ref{L1}. 
\end{proof}

\begin{lemm}\label{C1}
For any $\chi\notin\{\mu\nu,\mu\nu^{-1},\nu^{-(n-1)/2},\nu^{(n-1)/2}\}$, 
the representation $\W(\chi)$ is irreducible and dis\-tin\-gui\-shed. 
\end{lemm}

\begin{proof}
By Proposition \ref{Linked}, $\W(\chi)$ is irreducible. 
It satifies Condition (A) of Lemma \ref{TOL} with $k=n-2$, 
thus it is distinguished. 
\end{proof}

\begin{lemm}
\label{C2}
One has:
\begin{equation*}
\UIQ(\W(\mu\nu))=
\left\{
\begin{array}{ll}
\UIQ(\1_{n-2}\times\St_2\cdot\mu\nu^{1/2}) & \text{if $e>2$},\\
\UIQ(\1_{n-2}\times\1_2\cdot\mu\nu^{-1/2}) & \text{if $e=2$},
\end{array}
\right. \\
\end{equation*}
and $\UIQ(\W(\mu\nu^{-1}))=\UIQ(\1_{n-2}\times\1_2\cdot\mu\nu^{-1/2})$.  
\end{lemm}

\begin{proof}
First observe that, by Lemma \ref{zero}, $\W(\mu\nu)$ has 
$\1_{n-2}\times\St_2\cdot\mu\nu^{1/2}$ as a quotient if $e>2$ and 
$\W(\mu\nu^{-1})$ has $\1_{n-2}\times\1_2\cdot\mu\nu^{-1/2}$ as a quotient 
if $e\geq 2.$ Once again, applying Lemma \ref{zero} the statement is proved. 
\end{proof}

\begin{prop}
\label{C3}
Write $\Y(\mu)=\UIQ(\1_{n-2}\times\St_2\cdot\mu\nu^{1/2})$. 
Then:
\begin{equation*}
\Y(\mu)=
\left\{
\begin{array}{ll}
\1_{n-2}\times\St_2\cdot\mu\nu^{1/2} & \text{if $\mu\neq \nu^{-(n+1)/2}$ or $e=2$},\\
\LL_n^\vee & \text{if $\mu=\nu^{-(n+1)/2}$ and $e$ does not divide $n$ and $e>2$}.
\end{array}
\right. \\
\end{equation*}
\end{prop}

\begin{proof}
The statement follows from Proposition \ref{Linked} if $e=2$, 
and it follows from Proposition \ref{blm} if $\mu\neq\nu^{-(n+1)/2}$. 
Assume that $\mu=\nu^{-(n+1)/2}$ and $e$ does not divide 
$n$ and $e>2$. 
We have: 
\begin{eqnarray*}
\Y(\nu^{-(n+1)/2})&=&\UIQ(\W(\nu^{-(n-1)/2}))\\
&=&\UIQ(\nu^{-(n+1)/2}\times\1_{n-2}\times\nu^{-(n-1)/2})\\
&=&\UIQ(\nu^{-(n+1)/2}\times\nu_{n-1}^{-1/2})
\end{eqnarray*}
which is equal to $\LL_n^\vee$ by applying respectively Lemma 
\ref{C2}, Lemma \ref{PERM} (since $e$ does not divide $n$, the
representation $\1_{n-2}\times\nu^{-(n+1)/2}$ is irreducible by 
Proposition \ref{Linked}), Lemma \ref{L1} and \eqref{CONTRAL}. 
\end{proof}

\begin{prop}
\label{C4}
Write $\P(\mu)=\UIQ(\1_{n-2}\times\1_2\cdot\mu\nu^{-1/2})$. 
For $\mu\neq\nu^{-({n-3})/{2}}$, the representa\-tion $\P(\mu)$ is not 
distinguished, and we have:
\begin{equation*}
\P(\nu^{-({n-3})/{2}})=
\left\{
\begin{array}{ll}
\nu_{n-1}^{-1/2}\times\nu^{-({n-3})/{2}} & \text{if 
$e$ does not divide $n-2$ and $e>2$}, \\
\LL_n^\vee & \text{if 
$e=2$ and $n$ is odd}.
\end{array}
\right. \\
\end{equation*}
\end{prop}

\begin{proof}
The first assertion follows form Lemma \ref{TOL}.  
Assume now that $\mu=\nu^{-({n-3})/{2}}$. 
If $e>2$ does not divide $n-2$, then
$\P(\nu^{-({n-3})/{2}})=\UIQ(\W(\nu^{-({n-1})/{2}})$ by Lemma \ref{C2}. 
By Lemma \ref{L1} and Proposition \ref{Linked}, we have:
\begin{equation*}
\UIQ(\W(\nu^{-({n-1})/{2}})=\nu_{n-1}^{-1/2}\times\nu^{-({n-3})/{2}}. 
\end{equation*}
Assume now that $e=2$ and $n$ is odd.
By a similar argument as above, we deduce that:
\begin{equation*}
\P(\nu^{-({n-3})/{2}})=\UIQ(\nu^{-({n-3})/{2}}\times \nu_{n-1}^{-1/2}). 
\end{equation*}
By Proposition \ref{L4} and the observation following 
{Lemma} \ref{UIQ}, we get $\P(\nu^{-({n-3})/{2}})=\LL_n^{\vee}$. 
\end{proof} 

Note that $\1_{n-2}\times\mu\times\chi=\mu\times\1_{n-2}\times\chi$.
Thus:
\begin{eqnarray*}
\UIQ(\W(\nu^{(n-1)/2}))&=&
\left\{
\begin{array}{ll}
\UIQ(\mu\times\LL_{n-1}\cdot\nu^{-1/2}) & \text{if $e$ does not divide $n-1$},\\
\UIQ(\mu\times\nu_{n-1}^{-1/2}) & \text{if $e$ divides $n-1$},
\end{array}
\right. \\
\UIQ(\W(\nu^{-(n-1)/2}))&=&
\UIQ(\mu\times\nu_{n-1}^{-1/2}).
\end{eqnarray*}
We have the following proposition. 

\begin{prop}
\label{C5}
One has:
\begin{equation*}
\UIQ(\W(\nu^{-(n-1)/2}))=
\left\{
\begin{array}{ll}
\nu_{n-1}^{-1/2}\times\mu & \text{if $\mu\neq\nu^{-(n+1)/2}$},\\
\LL_n^\vee & \text{if $\mu=\nu^{-(n+1)/2}$ and $e$ does not divide $n$}. 
\end{array}
\right. \\
\end{equation*}
\end{prop}

\begin{proof}
This follows from Propositions \ref{L4} and \ref{L5}.
\end{proof}

It remains to study:
\begin{equation*}
\UIQ(\W(\nu^{(n-1)/2}))=\UIQ(\mu\times\LL_{n-1}\cdot\nu^{-1/2})
\end{equation*} 
when $e$ does not divide $n-1$. 
This will be done in Section \ref{LQuot}.

\section{Computing 
$\UIQ(\nu_{n-3}^{1/2}\times\tau\times\nu^{-(n-3)/2})$ 
for $\tau\in\GH_2$ infinite dimensional}
\label{GQuot}

In this section, we assume that $e>1$.
We consider all those infinite dimensional $\tau\in \GH_2$ such that
$\nu_{n-3}^{1/2}\times\tau$ is irreduci\-ble, that is:
\begin{enumerate}
\item
$\tau$ is cuspidal;
\item
$\tau$ is a Steinberg representation $\St_2\cdot\mu\nu^{1/2}$ with 
$\mu\notin\{\nu^{-(n-1)/2},\nu^{(n-1)/2}\}$ and $e>2$; 
\item
$\tau$ is a principal series $\l\times\mu$ with $\l\mu^{-1}\notin\{\nu^{-1},\nu\}$ and 
$\l,\mu\notin\{\nu^{-(n-3)/2},\nu^{(n-1)/2}\}$.
\end{enumerate}
In all these cases, we study the unique irreducible quotient:
\begin{equation}
\U(\tau)=\UIQ(\nu_{n-3}^{1/2}\times\tau\times\nu^{-(n-3)/2}).
\end{equation}
We first have the following results. 

\begin{lemm}
For all these $\tau$ as above, we have $\U(\tau)=\UIQ(\tau\times\1_{n-2})$.
\end{lemm}

\begin{proof}
It follows from the fact that 
$\nu_{n-3}^{1/2}\times\tau=\tau\times\nu_{n-3}^{1/2}$ and 
$\UIQ(\nu_{n-3}^{1/2}\times\nu^{-(n-3)/2})=\1_{n-2}$. 
\end{proof}

\begin{prop}
\label{Tcusp}
Assume that $\tau$ is cuspidal. 
Then $\U(\tau)=\tau\times\1_{n-2}$. 
\end{prop}

\begin{proof}
This follows from the fact that $\tau\times\1_{n-2}$ is irreducible when 
$\tau$ is cuspidal. 
\end{proof}

We now treat the cases where $\tau$ is not cuspidal.

\begin{prop}
\label{T2}
Assume 
$\tau=\l\times\mu$ with $\l\mu^{-1}\notin\{\nu^{-1},\nu\}$ and 
$\l,\mu\notin\{\nu^{-(n-3)/2},\nu^{(n-1)/2}\}$. 
Then we have:
\begin{equation*}
\U(\tau)=\l\times \mu \times \1_{n-2}
\quad
\text{for all $\l,\mu\neq\nu^{-(n-1)/2}$}
\end{equation*}
and, if $\mu=\nu^{-(n-1)/2}$ and $e$ does not divide $n-1$, 
then $\U(\tau)$ is not distinguished. 
\end{prop}

\begin{proof}
The first assertion follows from Proposition \ref{Linked}.  
Assume now that $\mu=\nu^{-(n-1)/2}$ and $e$ does not divide $n-1$. 
It follows from Proposition \ref{L4} that:
\begin{equation*}
\U(\tau)=\UIQ(\l\times \LL_{n-1}^{\vee}\cdot\nu^{1/2}),
\end{equation*}
which is not distinguished by Lemma \ref{TOL} with $k=1$.  
\end{proof}

\begin{prop}
\label{T1}
Assume $e>2$ and $\tau=\St_2\cdot\mu\nu^{1/2}$ with 
$\mu\notin\{\nu^{-(n-1)/2},\nu^{(n-1)/2}\}$.
Then:
\begin{equation*}
\U(\tau)=\tau\times \1_{n-2}
\quad
\text{for all $\mu\neq \nu^{-(n+1)/2}$}
\end{equation*}
and $\U(\tau)$ is not distinguished for $\mu=\nu^{-(n+1)/2}$. 
\end{prop}

\begin{rema}
If we assume that $e=2$ in Lemma \ref{T1},
then $\tau$ is cuspidal and this case has already been done in 
Lemma \ref{Tcusp}.
\end{rema}

\begin{proof}
Write $\tau=\L([0,1])\cdot\mu$.
By Proposition \ref{blm}, the representation $\tau\times \1_{n-2}$ is irreducible 
unless $\mu=\nu^k$ with $k$ a half-integer and 
the segments $[-(n-3)/2,(n-3)/2]$ and $[k,k+1]$ are juxtaposed, that is 
$\mu=\nu^{(n-1)/2}$ (which is not allowed) or $\mu=\nu^{-(n+1)/2}$.

Assume $\mu=\nu^{-(n+1)/2}$ and $e$ does not divide $n$
(thus $\mu\neq\nu^{(n-1)/2}$). 
Let $\L$ be the irreduci\-ble quotient of 
$\St_2\cdot\mu\nu^{1/2}\times \nu^{-(n-3)/2}$.
If $e>3$, note that $\St_3\cdot\nu^{-1}$ is the unique irreducible quotient of 
$\St_2\cdot \nu^{-3/2}\times \1$ (see p.~168 of \cite{P} and the exact
sequence (3.5) in \cite{V}). 
Twisting by $\nu^{-(n-3)/2}$, we see that $\L=\St_3\cdot\nu^{-(n-1)/2}$.
Moreover, by \cite[Theorem 2]{P} or \cite[Remark 6.7]{V},
no twist of $\L$ is distinguished. 
If $e=3$, $\L$ is equal to a twist of $\MM_3$, which is not 
distinguished by Lemma \ref{Sndist}. 
Hence, no twist of $\L$ is distinguished. 
Applying Lemma \ref{TOL} with $k=n-3$ to $\nu_{n-3}^{1/2}\times\L$ yields that 
it is not distinguished, and so $\U(\tau)$ is not distinguished.
\end{proof}

\section{The remaining cases}
\label{LQuot}

In this section, we assume that $e>1$ as in Sections \ref{TQuot} and \ref{GQuot}.
It remains for us to study the distinction of the following representations: 
\begin{enumerate}
\item
the irreducible quotients of
$\mu\times\LL_{n-1}\cdot\nu^{-1/2}$ 
for $\mu\in\GH_1-\{\nu^{-(n-1)/2},\nu^{(n-1)/2}\}$;
\item
the irreducible quotients of
$\LL_{n-1}^{\vee}\cdot \nu^{1/2}\times\chi$ for $\chi\in \GH_1$; 
\item
the irreducible quotient of 
$\LL_{n-1}\cdot\nu^{1/2}\times\nu^{-(n-3)/2}$.
\end{enumerate}
Note that we may assume $e$ does not divide $n-1$ 
(or else $\LL_{n-1}$ would be the trivial char\-ac\-ter).

The first case is the one that remains from Section \ref{TQuot},
the second one corresponds to Case 3 of Paragraph \ref{firstaid}
and the third one corresponds to the part of Case 4.e of Paragraph 
\ref{firstaid} which does not belong to Case 3. 

\subsection{Distinction of 
$\mu\times\LL_{n-1}\cdot\nu^{-1/2}$ and
$\LL_{n-1}^{\vee}\cdot \nu^{1/2} \times \chi$}
\label{LLntwists}
\label{Twolemmas}

In this paragraph, we show that, 
if $\LL_{n-1}^{\vee}\cdot \nu^{1/2} \times \chi$ is distinguished, 
then $\chi$ must be equal to $\nu^{-(n-3)/2}$. Given this, 
it will follow by Proposition \ref{GKext} that $\mu\times\LL_{n-1}\cdot\nu^{-1/2}$ 
is distinguished if and only if 
$\mu=\nu^{(n-3)/2}.$ 

\begin{lemm}
\label{LGL3}
Let $\chi\in\GH_1$ and $e>1$.
Then the representation $\St_2\cdot \nu^{-1/2}\times\chi$ 
is distinguished if and only if $\chi=1$. 
\end{lemm}

\begin{proof}
Write $\B(\chi)=\St_2\cdot \nu^{-1/2}\times \chi$.  
If $\chi=1$, then $\B(1)$ is distinguished as it satisfies (B) of Lemma 
\ref{TOL} for $k=2$. 

Assume $\chi\notin \{1,\nu,\nu^{-2}\}$. 
Since $\chi$ is nontrivial, Lemma \ref{TOL} implies that 
$\B(\chi)^{\vee}$ is not distingui\-shed.
Since $\chi\notin \{\nu, \nu^{-2}\}$, Lemma \ref{LinkedL} shows that 
$\B(\chi)$ is irreducible. 
Thus, by Lemma \ref{GK}, $\B(\chi)$ is not distinguished.
It remains to consider the case when 
$\chi\in \{\nu,\nu^{-2}\}$. 

If $e>3$, then we remind that $[\St_2\cdot\nu^{3/2}\times 1]= \St_3\cdot \nu + \LL_3 $ 
as in the complex case (see p.~168 of \cite{P} or (3.5) in \cite{V}). 
First we twist $\St_2\cdot\nu^{3/2}\times 1$ by $\nu^{-2}$. 
Secondly, 
we take the contragredient $\St_2\cdot \nu^{-3/2}\times 1$ and twist by $\nu$.
These yield:
\begin{equation*}
[\B(\nu^{-2})]=\LL_3 \cdot\nu^{-2}+\St_3 \cdot\nu^{-1}, 
\quad [\B(\nu)]=\LL_3^{\vee}\cdot\nu + \St_3 
\end{equation*}
respectively. None of these subquotients are distinguished.  

If $e=2$, then $\St_2$ is cuspidal, thus $\B(\nu)$ is irreducible and the
result follows from Lemma \ref{TOL}.

We finally assume that $e=3$.
We first claim the principal series $\xi=\nu^{-1}\times 1 \times\nu$
has length $7$, with subquotients:
\begin{equation*}
\1_3^{},\ \nu_3^{},\ \nu_3^{-1},\ \Pi_3,\ \Pi_3\cdot\nu,\ \Pi_3\cdot \nu^{-1}
\text{ and the cuspidal representation } \St_3.
\end{equation*}
Indeed, $\xi$ contains $\1_3$ and $\Pi_3$ as well as their twists by $\nu$ and
$\nu^2$, and it also contains the cuspidal (thus nondegenerate) representation 
$\St_3$ with multiplicity $1$.
The Jacquet module $\rp_{(1,1,1)}(\xi)$ has length $6$, thus our claim 
follows. 
Now we have:
\begin{eqnarray*}
{}[\xi] &=& [\nu_2^{-1/2}\times \nu]+[\B(\nu)] \\
&=& (\1_3+\nu_3^{-1}+\Pi_3)+[\B(\nu)]
\end{eqnarray*}
by Proposition \ref{P61}. 
It follows that:
\begin{equation*}
[\B(\nu)]=\nu_3+\MM_3+\MM_3\cdot\nu+\St_3
\end{equation*}
in the Grothendieck group of finite length representations of $\G_3$.

By Lemma \ref{Sndist} and Theorem \ref{THEOCUSP}, 
none of these subquotients are distinguished. 
Since $\B(\nu^{-2})$ is equal to $\B(\nu)$, our lemma is proved. 
\end{proof}

Given $\chi\in\GH_1$, we now write:
\begin{equation*}
\A(\chi)=\LL_{n-1}^{\vee}\cdot\nu^{1/2}\times \chi.
\end{equation*} 
We study the distinction of $\A(\chi)$ in the following lemma. 

\begin{lemm}
\label{C6}
\label{C7}
Assume that $e$ does not divide $n-1$, 
and let $\chi\in\GH_1$.
Then $\A(\chi)$
is distingui\-shed if and only if $\chi=\nu^{-(n-3)/2}$. 
\end{lemm}

\begin{proof}
First, Lemma \ref{TOL} with $k=n-1$ shows that 
$\A(\nu^{-(n-3)/2})$ is distin\-guish\-ed.

For the converse, we may assume that $n\geq 4$ since we have treated the case
when $n=3$ in Lemma \ref{LGL3}.
Assume first that $e>2$. By Proposition \ref{C3}, $\A(\chi)$ is a quotient of:
\begin{equation*}
\nu_{n-3}^{1/2}\times\St_2^{}\cdot\nu^{-({n-2})/{2}}\times\chi,
\end{equation*}
which is distinguished by Remark \ref{TOLext} if and only if Condition (A) of 
Lemma \ref{TOL} is satis\-fied with $k=n-3$.
This is the case if and only 
if $\St_2\cdot\nu^{-1/2}\times\chi\nu^{(n-3)/2}$ is distinguished.
By Lemma \ref{LGL3}, this happens if and only if $\chi=\nu^{-(n-3)/2}$.

Assume now that $e=2$.
Note that the characters $\nu^{(n-1)/2}$ and $\nu^{(n+1)/2}$ are the only
ones that are obtained from $\nu^{-(n-3)/2}$ up to a translation of an integer 
power of $\nu$. 

Assume first that $\chi\notin \{\nu^{-(n-1)/2},\nu^{(n-1)/2}\}$. 
Then $\A(\chi)$ is irreducible by Proposition \ref{Lanzmann},
and Lemma \ref{TOL} implies that $\A(\chi)^{\vee}$ is not distinguished. 
By Proposition \ref{GK}, $\A(\chi)$ is not distinguished either. 

It remains to consider the case where $\chi=\nu^{-(n-1)/2}=\nu^{(n+1)/2}$. 
We write $\A=\A(\nu^{(n+1)/2)})$. 
By definition, $\LL_{n-1}^{\vee}\cdot \nu^{1/2}$ is the unique irreducible 
quotient of $\nu^{(n+1)/2}\times \1_{n-2}$. 
The representa\-tion $\A$ is thus a quotient of 
$\V=\nu^{(n+1)/2}\times \1_{n-2} \times \nu^{(n+1)/2}$. 
Now write the two exact sequences:
\begin{equation}
\label{A} 
0\to \nu_{n-1}^{-1/2}\to \nu^{(n+1)/2}\times \1_{n-2}\to 
\LL_{n-1}^{\vee}\cdot \nu^{1/2}\to 0
\end{equation} 
and:
\begin{equation}
\label{B} 
0\to \LL_{n-1}^{\vee}\cdot \nu^{1/2} \to \1_{n-2} 
\times \nu^{(n+1)/2}\to\nu_{n-1}^{-1/2} \to 0.
\end{equation} 
Computing $\eqref{A}\times\nu^{(n+1)/2}$, we get:
\begin{equation*}
0\to \W \to \V \ffr{\a} \A \to 0 
\end{equation*}
where $\W$ is the representation $\nu_{n-1}^{-1/2}\times\nu^{(n+1)/2}$, 
which is irreducible 
since $\nu^{(n+1)/2}\neq\nu^{(n-1)/2}$ and $\nu^{(n+1)/2}\neq \nu^{-(n+1)/2}$. 
Thus $\W$ is isomorphic to $\nu^{(n+1)/2}\times\nu_{n-1}^{-1/2}$. 
Computing $\nu^{(n+1)/2}\times\eqref{B}$ we get: 
\begin{equation*}
0\to\nu^{(n+1)/2}\times \LL_{n-1}^{\vee}\cdot \nu^{1/2} \to \V \ffr{\b} \W \to 0. 
\end{equation*}
Observe that $\W$ is distinguished by Lemma \ref{TOL}, 
thus $\V$ is also distinguished. 
Lemma \ref{TOL} (applied with $k=n-1$) also shows that the space of 
$\H_n$-invariant forms on $\V$ is one-dimensional. 

Now we claim $\A$ is not distinguished. 
Assume $\A$ is distinguished, 
and let $\T$ denote a nonzero invariant linear form on $\V$ which is trivial 
on $\K_1=\Ker(\a)$.
Since $\V$ has a one-dimensional space of invariant forms, 
$\T$ is proportional to any nonzero invariant linear form on $\V$ which is 
trivial on $\K_2=\Ker(\b)$.
Thus, $\T$ is zero on $\K_1+ \K_2$. 
Since $\T$ is nonzero, $\K_1+\K_2$ is different from the whole space $\V$. 
Since $\K_1$ is irreducible and isomorphic to $\W$, 
we get that $\K_1+\K_2=\K_2$, thus:
\begin{equation*}
\K_1 \subseteq \K_2 \simeq \nu^{(n+1)/2}\times \LL_{n-1}^{\vee}\cdot \nu^{1/2}.
\end{equation*}
It follows that:
\begin{equation*}
\W = \UIS(\nu^{(n+1)/2}\times \LL_{n-1}^{\vee}\cdot \nu^{1/2})
= \UIQ(\A).
\end{equation*}
Thus $\W\cdot \nu$ is the unique irreducible quotient of $\A\cdot \nu$.
Observe that $\W\cdot \nu= \nu_{n-1}^{1/2}\times \nu^{(n+3)/2}$ is isomorphic to
$\W^{\vee}$ and hence is distinguished by Proposition \ref{GK}. 
However, the representation 
$\A\cdot \nu= \LL_{n-1}^{\vee} \cdot \nu^{-1/2}\times \nu^{(n-1)/2}$
is not distinguished by Lemma \ref{TOL}, a contradiction. 
\end{proof}

\subsection{Distinction of 
$\UIQ(\LL_{n-1}^{\vee}\cdot \nu^{1/2} \times \nu^{-(n-3)/2})$ and 
$\UIQ(\nu^{(n-3)/2}\times\LL_{n-1}\cdot\nu^{-1/2})$}
\label{Q1}

By Lemma \ref{C6}, 
in order to finish Cases 1 and 2 of Section 12 for $e>1$,
it remains to discuss the distinction of the irreducible quotients:
\begin{equation*}
\UIQ(\LL_{n-1}^{\vee}\cdot\nu^{1/2}\times \nu^{-(n-3)/2})
\quad\text{and}\quad
\UIQ(\nu^{(n-3)/2}\times\LL_{n-1}\cdot\nu^{-1/2}).
\end{equation*}
Note that the latter is the contragredient of the former,
thus it is enough to study the distinction of the first one. 
Moreover, if $n=3$, then:
\begin{equation*}
\UIQ(\LL_{2}^{\vee}\cdot\nu^{1/2}\times\1) = \St_2\cdot \nu^{-1/2}\times 1
\end{equation*}
is distinguished by Lemmas \ref{LinkedL} and \ref{LGL3}. 
So we will assume that $n\geq 4$ in the remainder of this Section. 
In what follows, the computation  of distinguished quotients will fall into three cases:
\begin{enumerate}
\item
$e>2$ and $e$ does not divide $n-2$;
\item
$e>2$ and $e$ divides $n-2$ (this implies that $e$ does not divide $n$);
\item
$e=2$ (this implies that $e$ divides $n-2$ since $e$ does not divide $n-1$). 
\end{enumerate}
We start with the following lemma, which follows from Lemma \ref{L1}.

\begin{lemm}\label{appl0} 
Assume $e>1$. 
We have:
\begin{equation*}
[\nu_{n-1}^{-1/2}\times \nu^{-(n-3)/2}] = 
\left\{
\begin{array}{ll} 
\nu_{n-1}^{-1/2}\times \nu^{-(n-3)/2} & 
\text{if $e>2$ and $e$ does not divide $n-2$}, \\
\1_n+\Pi_n \cdot \nu^{-1} & \text{if $e>2$ and $e$ divides $n-2$}, \\
\nu_n^{-1}+ \Pi_n^{\vee} & \text{if $e=2$ and $e$ does not divide $n-2$}, \\
\1_n^{}+\nu_n^{-1}+\Pi_n^{\vee} & \text{if $e=2$ and $e$ divides $n-2$.}
\end{array}
\right.
\end{equation*}
\end{lemm}

We now define two irreducible representations of $\G_n$. 

\begin{defi}
Assume $e>1$ and $n\> 4$.
Define:
\begin{eqnarray*}
\Phi_n &=& \Z\(\left[-\frac{n-3}{2}, 
\frac{n-3}{2}\right]+\left[-\frac{n-1}{2},-\frac{n-3}{2}\right]\), \\
\Psi_n &=& \Z\(\left[-\frac{n-3}{2}, \frac{n-3}{2}\right]+
\left[-\frac{n+1}{2},-\frac{n-1}{2}\right]\).
\end{eqnarray*}
\end{defi}

Observe that $\Phi_n$ is selfdual if $e$ divides $n-2$
and $\Psi_n$ is selfdual if $e$ divides $n$. 
We also recall that $\Pi_n^{\vee}=\Pi_n^{}\cdot \nu^{-1}$ if $e$ divides $n$.

\begin{lemm}
\label{appl1} 
Assume $e>1$ and $n\geq4$,
and suppose $e$ does not divide $n-1$.
The irreducible subquotients of:
\begin{equation}
\label{ind11}
\1_{n-2}\times \1_2 \cdot \nu^{-(n-2)/2} = 
\Z\(\left[-\frac{n-3}{2},\frac{n-3}{2}\right]\) \times 
\Z\(\left[-\frac{n-1}{2},-\frac{n-3}{2}\right]\)
\end{equation}
are:
\begin{enumerate}
\item 
the representations $\nu_{n-1}^{-1/2}\times \nu^{-(n-3)/2}$ and $\Phi_n$
if $e$ does not divide $n-2$, 
\item 
the representations $\1_n^{}$, $\Pi_n^{\vee}\cdot\nu$, $\Pi_n^{}\cdot \nu^{-1}$ 
and $\Phi_n^{}$ if $e$ divides $n-2$. 
\end{enumerate} 
Moreover, all subquotients appear with multiplicity 1 if $e>2$. 
If $e=2$, only $\1_n$ 
may appear with multiplicity more than 1. 
\end{lemm}

\begin{proof}
We apply Proposition \ref{length2seg}. 
The irreducible subquotients $\Phi_n$ and: 
\begin{equation*}
\Z\(\left[-\frac{n-1}{2},\frac{n-3}{2}\right]+\left[-\frac{n-3}{2}\right]\) = 
\left\{
\begin{array}{ll} 
\nu_{n-1}^{-1/2}\times \nu^{-(n-3)/2} & 
\text{if $e$ does not divide $n-2$}, \\
\Pi_{n}\cdot\nu^{-1} & \text{if $e$ divides $n-2$},
\end{array}
\right.
\end{equation*}
always occur in \eqref{ind11}. 
The irreducible subquotients $\Z([-(n-3)/2,(n-1)/2]+[-(n-1)/2])$ and 
$\Z([-(n-1)/2,(n-1)/2])=\1_n$ occur if and only if $e$ divides $n-2$.
The irreducible subquotients $\Z([-(n-3)/2,(n+1)/2])=\nu_n^{}$ and 
$\Z([-(n+1)/2,(n-3)/2])=\nu_n^{-1}$ do not occur, since $e$ 
does not divide $n-1$ and $e>1$.
\end{proof}

Similarly, by applying Proposition \ref{length2seg}, we have the following. 

\begin{lemm}
\label{appl3}
Assume $e>1$ and $n\geq 4$, 
and suppose $e$ does not divide $n-1$.
The irreducible subquotients of $\1_{n-2}\times \1_2\cdot \nu^{-n/2}$ are:
\begin{enumerate}
\item 
the representations $\nu_n$ and $\Psi_n$ if $e$ does not divide $n$, 
\item 
the representations $\nu_n^{}$, $\nu_n^{-1}$ and $\Psi_n^{}$ if $e$ divides $n$.
\end{enumerate}
Moreover, all subquotients appear with multiplicity 1 if $e>2$. 
If $e=2$, only $\nu_n$ may appear with multiplicity more than 1.
\end{lemm}

\begin{lemm}\label{Ln1easy}
Assume $e$ does not divide $n-2$ nor $n-1$, thus $e>2$. 
Then:
\begin{equation*}
\UIQ(\LL_{n-1}^{\vee}\cdot \nu^{1/2}\times \nu^{-(n-3)/2}) 
= \1_{n-2}\times \St_2\cdot \nu^{-(n-2)/2}.
\end{equation*}
\end{lemm}

\begin{proof}
Since $e$ does not divide $n-2$, the product 
$\1_{n-2} \times\nu^{-(n-3)/2}$ is irreducible, 
thus it is isomorphic to $\nu^{-(n-3)/2}\times \1_{n-2}$. 
Moreover, the representation
$\nu^{-(n-1)/2}\times \nu^{-(n-3)/2}\times \1_{n-2}$ has a 
unique irre\-ducible quotient by Proposition \ref{UIQ}. 
It follows that this unique irre\-ducible quotient is
$\St_2\cdot \nu^{-(n-2)/2}\times\1_{n-2}$, 
which is irreducible by Proposition \ref{blm}. 
\end{proof}

\begin{lemm}
\label{Ln1}
Assume $e>1$ and $n\geq 4$,
and suppose $e$ does not divide $n-1$.
If the re\-pre\-sentation $\UIQ(\LL_{n-1}^{\vee}\cdot \nu^{1/2}\times \nu^{-(n-3)/2})$ 
is distinguished, then it is either $\1_n$ or 
$\1_{n-2}\times \St_2\cdot \nu^{-(n-2)/2}$. 
\end{lemm}

\begin{proof}
If $e>2$ and does not divide $n-2$ we reduce to the case
of Lemma \ref{Ln1easy}. Therefore, we need
only consider either $e=2$ or $e$ divides $n-2.$
The representation $\B=\LL_{n-1}^{\vee}\cdot \nu^{1/2}\times \nu^{-(n-3)/2}$ 
is a quotient of $\U=\nu^{-(n-1)/2}\times\1_{n-2} \times \nu^{-(n-3)/2}$. 
Observe that we have $[\U]= [\P] + [\B]$ where $\P=\nu_{n-1}^{-1/2}\times \nu^{-(n-3)/2}.$ 
Now $\U$ has the same semisimplification as 
$\1_{n-2} \times \nu^{-(n-1)/2}\times \nu^{-(n-3)/2}$,
thus we have:
\begin{eqnarray*}
[\U]&=&
\left\{
\begin{array}{ll}
\1_{n-2}\times \St_2\cdot \nu^{-(n-2)/2} + [\1_{n-2}\times \1_2\cdot\nu^{-(n-2)/2}] 
& \text{if $e>2$},\\
\1_{n-2}\times \St_2\cdot \nu^{-(n-2)/2} + [\1_{n-2}\times \1_2 \cdot\nu^{-(n-2)/2}] + 
[\1_{n-2}\times \1_2\cdot \nu^{-n/2}] & \text{if $e=2$},
\end{array}
\right. 
\end{eqnarray*}
since $\1_{n-2}\times \St_2\cdot\nu^{-(n-2)/2}$ is irreducible by Proposition \ref{BLM01}.
Since $e$ does not divide $n-1,$ the irreducible subquotients occurring in 
$\1_{n-2}\times \1_2\cdot \nu^{-(n-2)/2}$ by Lemma \ref{appl1} are: 
\begin{equation*}
\1_n^{},
\
\nu^{}_n,
\
\MM^{}_n\cdot \nu^{-1},
\
\MM_n^{\vee}\cdot \nu,
\
\Phi_n^{}.
\end{equation*}
Moreover, all of them occur with multiplicity $1$ except $\1_n$, 
which may appear with larger multiplicity if $e=2$. 
Also, By Lemma \ref{appl3}, since $e$ does not divide $n-1$, 
the irreducible subquotients occurring in 
$\1_{n-2}\times \1_2\cdot \nu^{-n/2}$ are $\nu_n^{}$, $\nu_n^{-1}$ and $\Psi_n$.
Here $\Psi_n$ always occurs with multiplicity one and if $e=2$ 
the other factors might appear with larger multiplicity. 
We will now obtain $[\B]$ by comparing $[\U]$ obtained 
from the two different expressions for $[\U]$ above. 

Assume that $e>2$ and $e$ divides $n-2$. Then by Lemma \ref{appl0} we have 
$[\P]=\1_n + \MM_n\cdot \nu^{-1}$. 
Hence we have:
\begin{equation*} 
[\B]=\Phi_n^{} + \MM_n^{\vee} + \1_{n-2}^{} \times \St_2 \cdot \nu^{-(n-2)/2}.
\end{equation*} 

Next assume that $e=2$. 
Then $e$ necessarily divides $n-2$ (since $e$ does not divide $n-1$). 
By Lemma \ref{appl0}, we have 
$[\P]= \1_n^{} + \nu_n^{} + \MM_n^{\vee}$. 
Recall that if $e$ divides $n$ then $\MM_n^{\vee}= \MM_n \cdot \nu^{-1}$. 
Hence the only possible irreducible subquotients of $\B$ are:
\begin{equation*} 
\1_n, \ \nu_n, \ \1_{n-2} \times \St_2 \cdot \nu^{-(n-2)/2}, 
\ \MM_n^{\vee} \cdot \nu, \ \Phi_n, \ \Psi_n. 
\end{equation*}
The proof of Lemma \ref{Ln1} will be complete if we prove the following lemma.

\begin{lemm}\label{Z1Z2} 
Assume $e>1$ and $n\geq 4$,
and suppose that $e$ does not divide $n-1$.
For any cha\-racter $\chi\in \GH_1$, the twists
$\Phi_n\cdot \chi$ and $\Psi_n\cdot \chi$ are not distinguished. 
\end{lemm}

\begin{proof}
Observe that $\Phi_n\cdot \chi$ and $\Psi_n \cdot \chi$
have only first and second derivatives which are nonzero. 
Thus we will use Lemma \ref{BZapp}. 

Assume the first derivative of $\Phi_n\cdot\chi$ has a 
quotient isomorphic to $\nu_{n-1}^{-1/2}$.
By Lemma \ref{Derivative}, this would imply that $\Phi_n\cdot\chi$ is a 
character, or that the multisegment that corresponds to it is made of 
one segment of length $n-1$ and one of length $1$, which is not the case. 
The same argument holds for $\Psi_n\cdot\chi$.

From Lemma \ref{appl3}, we see that the 
second derivative of $\Psi_n\cdot\chi$ is 
$(\nu_{n-3}^{-1/2}\times\nu^{-(n+1)/2})\cdot\chi$,
and since $e$ does not divide $n-1$ it is irreducible 
for all $\chi\in \GH_1$ by Lemma \ref{L1}. 
Thus it does not have any character as a quotient.
Now we have:
\begin{eqnarray*}
[\1_{n-2}\times \1_2\cdot \nu^{-(n-2)/2}]^{(2)} &=&
[\nu_{n-3}^{-1/2}\times\nu^{-(n-1)/2}] \\
&=& \left\{
\begin{array}{ll}
 \nu_{n-2}^{-1}+\1_{n-2}^{}+\MM_{n-2}^{\vee} & \text{if $e$ divides $n-2$}.\\
 \nu_{n-2}^{-1}+\MM_{n-2}^{\vee} & \text{if $e$ does not divide $n-2$}.
\end{array}
\right.
\end{eqnarray*}
By Lemma \ref{derPi} and Corollary \ref{derPicor}, we have 
$(\MM_n\cdot \nu^{-1})^{(2)}=\nu_{n-2}$ and
$(\MM_n^{\vee}\cdot \nu)^{(2)}=\1_{n-2}$. 
Therefore, we conclude using Lemma
\ref{appl1} that the second derivative of $\Phi_n$ is $\MM_{n-2}^{\vee}$.
By Lemma \ref{BZapp}, $\Phi_n\cdot \chi$ and $\Psi_n\cdot \chi$ 
are not distinguished. 
\end{proof}

This ends the proof of Lemma \ref{Ln1}.
\end{proof}

\subsection{Distinction of $\UIQ(\LL_{n-1}\cdot \nu^{1/2}\times \nu^{-(n-3)/2})$}
\label{Q2}

We begin this paragraph 
with a simple lemma which we will need in the sequel. 
We remind that $n\>4$ and $e$ does not divide $n-1$.

\begin{lemm}
\label{cruc}
Let $n\geq 4$. Assume that $e>1$ and let 
$\l,\mu\in\GH_1-\{\nu^{-(n-3)/2}\}$.
Then the induced representation 
$\1_{n-2}\times \l\times \mu$ has a unique irreducible quotient. 
\end{lemm}

\begin{proof}
If $\l=\mu$, the result follows from \cite[Lemma 6.1]{MSu}. 
We thus assume that $\l\neq \mu$. 
By the geometric lemma, the semi-simplification of the Jacquet module 
$\rp_{(n-2,1,1)}(\1_{n-2}\times \l \times \mu)$ is the sum of the 
following representations:
\begin{enumerate}
\item
$\1_{n-2}\otimes \l \otimes \mu$,
\item
$\1_{n-2}\otimes \mu \otimes \l$,
\item
$[\nu_{n-3}^{-1/2}\times \l] \otimes \nu^{(n-3)/2} \otimes \mu$,
\item
$[\nu_{n-3}^{-1/2}\times \mu] \otimes \nu^{(n-3)/2}\otimes \l$,
\item
$[\nu_{n-3}^{-1/2}\times \l] \otimes \mu \otimes \nu^{(n-3)/2}$,
\item
$[\nu_{n-3}^{-1/2}\times \mu] \otimes \l \otimes \nu^{(n-3)/2}$,
\item
$[\nu_{n-4}^{-1}\times \l \times \mu] \otimes \nu^{(n-5)/2}\otimes \nu^{(n-3)/2}$,
\end{enumerate} 
in the Grothendieck group of finite length representations of the Levi 
subgroup $\G_{n-2}\times\G_1\times\G_1$.
If $\l,\mu\neq\nu^{(n-3)/2}$ then by \cite[Lemme 2.4]{MSc} the representation
$\1_{n-2}\times \l\times \mu$ has a unique irreduci\-ble subrepresentation.
The result follows by taking contragredients.
\end{proof}

\begin{lemm}
\label{Ln2easy}
Assume that $e>2$ and $e$ does not divide $n-2$. 
Then:
\begin{equation*}
\UIQ(\LL_{n-1}\cdot \nu^{1/2}\times \nu^{-(n-3)/2})=\LL_n.
\end{equation*}
\end{lemm}

\begin{proof}
The representation $\C=\LL_{n-1}\cdot \nu^{1/2}\times\nu^{-(n-3)/2}$ 
is a quotient of:
\begin{equation*}
\W=\nu_{n-2}\times\nu^{(n+1)/2}\times\nu^{-(n-3)/2}.
\end{equation*}
If we apply Lemma \ref{cruc} with $\l=\nu^{(n-1)/2}$ and $\mu=\nu^{-(n-1)/2}$,
which is possible since $e>2$ and $e$ does not divide $n-2$,
we deduce that $\W\cdot\nu^{-1}$ (thus $\W$) has a unique irreducible 
quotient. 
Since $\nu^{(n+1)/2}\times\nu^{-(n-3)/2}$ is irreducible, it is isomorphic to 
$\nu^{-(n-3)/2}\times\nu^{(n+1)/2}$.
Thus $\nu_{n-1}^{1/2}\times\nu^{(n+1)/2}$ is a quotient 
of $\W$, and it has the unique irreducible quotient $\LL_n$. 
\end{proof}

\begin{lemm}\label{Ln2}
Assume that $e>1$ and $n\geq 4$. 
If $\UIQ(\LL_{n-1}\cdot \nu^{1/2}\times \nu^{-(n-3)/2})$ 
is distinguished, then it is $\LL_n$. 
\end{lemm}

\begin{proof}
If $e>2$ and does not divide $n-2$ we reduce to the case 
of Lemma \ref{Ln2easy}.
We may assume that $e=2$ or $e$ divides $n-2.$ In this proof,  
$\W,\C$ are  as in  Lemma \ref{Ln2easy} and $\U,\P$ are as in Lemma \ref{Ln1}. 
Assume that $e$ divides $n-2.$ Then $\nu^{(n-1)/2}=\nu^{-(n-3)/2}$ and therefore
\begin{equation*}
\W\cdot\nu^{-1}= \1_{n-2} \times \nu^{-(n-3)/2} \times \nu^{-(n-1)/2}.
\end{equation*}
Therefore, we have 
$$[\W]=[\U\cdot \nu] \quad\text{~and~}\quad[\W]= [\P^{\vee}\cdot \nu] + [\C]$$
where $\P^{\vee}\cdot \nu = \nu_{n-1}^{3/2}\times \nu^{(n-1)/2}.$ 

If $e>2$ and $e$ divides $n-2$, 
then we twist the subquotients of $\U$ in the proof of Lemma \ref{Ln1} by 
$\nu$ to get: 
\begin{equation*} 
[\W]= \nu_{n-2} \times \St_2 \cdot \nu^{-(n-4)/2} + \Phi_n\cdot \nu + \Pi_n + 
\Pi_n\cdot \nu^2 + \nu_n 
\end{equation*}
and $[\P^{\vee}\cdot \nu]= \nu_n+ \Pi_n\cdot \nu^{2}$. 
It follows that:
\begin{equation*} 
[\C]= \nu_{n-2} \times \St_2 \cdot \nu^{-(n-4)/2} + \Phi_n\cdot \nu + \Pi_n.
\end{equation*}
Hence the only distinguished subquotient is $\Pi_n$ which is the definition of $\LL_n$ when $e$ does not divide $n.$

Now $e=2$, which necessarily divides $n-2$. 
Then $\P^{\vee}\cdot \nu$ is isomorphic to $\P$.
We twist the subquotients of $\U$ in the proof of Lemma \ref{Ln1} by $\nu$ to 
conclude that the only possible irreducible subquotients of $\W$ are:
\begin{equation*}
\1_n, \ \nu_n, \ \MM_n, \ \MM_n^{\vee}, \ \Phi_n \cdot \nu, \ \Psi_n \cdot \nu
\end{equation*}
with all representations except possibly $\1_n$ and $\nu_n$ appearing with 
multiplicity greater than $1$. 
Since $[\P]= \1_n + \nu_n + \MM_n^{\vee}$ it follows that the only possible 
irreducible subquotients of $\C$ are: 
\begin{equation*}
\1_n, \ \nu_n, \ \MM_n, \ \Phi_n \cdot \nu, \ \Psi_n \cdot \nu.
\end{equation*}
Hence the only distinguished subquotient is $\1_n$ which is the definition of 
$\LL_n$ when $e$ divides $n.$ 
This completes the proof of the Lemma. 
\end{proof}

\begin{rema}
\label{REMARQUEFINALE}
Given $\pi\in \GH_n$, we write $d(\pi)$ for the dimension of 
$\Hom_{\H_n}(\pi,\R)$. 
We notice that, unlike in the complex case \cite{AGS}, when $\R$ is of positive 
characteristic, it is not true in general that $d(\pi)\<1$ for all 
$\pi\in\GH_n$. 
For example, when $n=2$ and $e=1$, we have 
constructed $\pi\in\GH_n$ with $d(\pi)=2$
(see Theorem \ref{gl2mult}). 
More generally, when $e=1$ and $\ell$ divides $n$, we have 
shown (see Remark \ref{Pi_nmult2}) that $d(\Pi_n)=2$. 
However, when $e>1$ we expect that $d(\pi)\<1$ for all $\pi\in\GH_n$, 
as in the case $n=2$. 
It is interesting to note the analogy of the situation in the case of $e=1$ of 
Theorem \ref{gl2mult} with \cite[Corollary 3.3]{YBS}, where the author shows 
that $d(\pi)\<2$ where $\G=\GL_n(\mathbb{F}_q)$ and $\R$ is an 
algebraically closed field of characteristic prime to $2$ and $q$. 
\end{rema} 

\providecommand{\bysame}{\leavevmode ---\ }

\end{document}